\documentclass[12pt]{article}
\usepackage[margin=1in]{geometry}
\usepackage{amsmath,amsthm,amssymb}
\usepackage{verbatim}
\usepackage{tikz}
\usepackage{mathtools}

\tikzstyle{vertex}=[draw,thick,fill=white,circle,inner sep=2pt]
\tikzstyle{tiny_vertex}=[draw,thick,fill=white,circle,inner sep=1.5pt]
\tikzstyle{full}=[draw,thick,fill=black,circle,inner sep=2pt]
\tikzstyle{empty}=[draw,color=black!40!white,thick,fill=white,circle,inner sep=2pt]

\newtheorem{thm}{Theorem}[section]
\newtheorem{cor}[thm]{Corollary}
\newtheorem{lem}[thm]{Lemma}
\newtheorem{prop}[thm]{Proposition}

\newtheorem{obs}[thm]{Observation}

\theoremstyle{definition}

\newtheorem{defn}[thm]{Definition}

\allowdisplaybreaks

\DeclareMathOperator{\dist}{dist}
\DeclareMathOperator{\spec}{spec}
\DeclareMathOperator{\sign}{sign}
\DeclareMathOperator{\diam}{diam}

\newcommand{\Dq}{\mathcal{D}_q}
\newcommand{\sq}{\, \square \,}

\newcommand{\fsp}{\,\,\,\,\,}
\newcommand{\vxi}[2]{\!\!\!\underset{(#1,#2)}{\circ}\!\!\!}

\title{Spectral properties of the exponential distance matrix}
\author{Steve Butler\footnote{Iowa State University, Ames, IA, USA \texttt{\{butler,lorenkj\}@iastate.edu}}
\and Elizabeth Cooper\footnote{Oberlin College, Oberlin, OH, USA \texttt{pcooper@oberlin.edu}}
\and Aaron Li\footnote{Carleton College, Northfield, MN, USA \texttt{lia2@carleton.edu}}
\and Kate Lorenzen\footnotemark[1]
\and Zo\"e Schopick\footnote{Macalester College, Saint Paul, MN, USA \texttt{zschopic@macalester.edu}}}
\date{\empty}

\begin{document}
\maketitle

\begin{abstract}
Given a graph $G$, the exponential distance matrix is defined entry-wise by letting the $(u,v)$-entry be $q^{\dist(u,v)}$, where $\dist(u,v)$ is the distance between the vertices $u$ and $v$ with the convention that if vertices are in different components, then $q^{\dist(u,v)}=0$. In this paper, we will establish several properties of the characteristic polynomial (spectrum) for this matrix, give some families of graphs which are uniquely determined by their spectrum, and produce cospectral constructions.

\end{abstract}

\section{Introduction}
A simple graph $G$ is a set of vertices, $V(G)$, and a set of edges, $E(G)$, that connect distinct vertices to each other.  A graph can be represented by a matrix in multiple ways.  Once we have a matrix $M$ we can then look at the characteristic polynomial, or the set of eigenvalues (spectrum, denoted $\spec(M)$), of the matrix. Spectral graph theory is the study of information that connects a graph and the spectrum of the matrix $M$.

The most studied matrix associated with a graph $G$ is the \emph{adjacency matrix}, $A^G=A$, whose entries are defined as follows. 
\begin{equation*}
    A_{u,v} = 
    \begin{cases}
          1 & \text{if }u \sim v,\\
          0 & \text{otherwise},\\
    \end{cases}
\end{equation*}
where $u\sim v$ indicates that $u$ and $v$ are adjacent \cite{CDS}.

In addition to the adjacency matrix, many other matrices have been explored.  Among them is the \emph{distance matrix}, $\mathcal{D}^G=\mathcal{D}$, whose entries are defined as follows.
\begin{equation*}
    \mathcal{D}_{u,v} = \dist_G(u,v),
\end{equation*}
where $\dist_G(u,v)$ is the distance between vertices $u$ and $v$ in $G$.  For vertices in two distinct components the usual convention is to have $\dist_G(u,v)=\infty$.  As a result, the distance matrix is almost always only considered for connected graphs.  Extensive lists of properties and results for the distance matrix can be found in a survey \cite{AH}.

Variations of the distance matrix have been considered.  Bapat, Lal and Pati \cite{bapat} introduced the exponential distance matrix, $\Dq^G=\Dq$, whose entries are defined as follows.
\begin{equation*}
(\Dq)_{u,v}=\begin{cases}
q^{\dist_G(u,v)}&\text{if $u$ and $v$ in same component,}\\
0&\text{otherwise},
\end{cases}
\end{equation*}
where $q$ is a variable.  If we restrict $q$ to the interval $(-1,1)$, then the definition can simplify to $(\Dq)_{u,v} = q^{\dist(u,v)}$ since by convention $q^\infty=0$.  This allows us to work with disconnected graphs.

Previous work for the exponential distance matrix \cite{bapat,YanYeh} has been restricted to trees.  However, this matrix has many rich and interesting properties that we will explore in this paper.  In Section~\ref{sec:ops}, we will look at graph operations as they relate to $\Dq$ and connect the spectrum of $\Dq$ to the spectra of $A$ and $\mathcal{D}$ in certain situations.  In Section~\ref{sec:poly}, we look at some information about the graph that can be derived from the characteristic polynomial of $\Dq$ and use this to show how some families are determined by their spectrum.  In Section~\ref{sec:cospec}, we give several constructions for cospectral graphs. Finally, in Section~\ref{sec:future}, we will give concluding remarks and open problems.

\section{Graph operations and connections to other spectra}\label{sec:ops}

In this section we develop tools to determine the spectrum of the exponential distance matrix and then use these tools to find the spectrum of certain families of graphs. 

\subsection{Graph operations}
We begin by examining some basic graph operations, as many useful graph families can be constructed via graph operations.


\begin{prop}\label{union}
    Let $G \cup H$ denote the disjoint union of $G,H$. Then \[\spec(\Dq^{G \cup H})= \spec(\Dq^G) \cup \spec(\Dq^H).\]
\end{prop}

\begin{proof}
     The exponential distance matrix of $G \cup H$, denoted $\Dq^{G \cup H}$, can be written in the following block form 
     \begin{align*}
         \Dq^{G \cup H} = 
         \left(\begin{array}{cc}
            \Dq^G
            & O\\
            O &
            \Dq^H
        \end{array}\right).
     \end{align*}
     Since the eigenvalues of a block diagonal matrix are the eigenvalues of the blocks themselves, the spectrum of $\Dq^{G \cup H}$ is the union of the spectra of $\Dq^G$ and $\Dq^H$.
\end{proof}

The result of the union operation is what would be expected. We now consider an operation that has an unexpected result based on what happens with the adjacency and distance matrices. 

\begin{defn}
    The \emph{Cartesian product} of two graphs $G,H$, denoted $G \sq H$, has vertices $V(G \sq H) = \{(u,v) \, \mid \, u \in V(G), v \in V(H)\}$ and $(u_1, v_1) \sim (u_2, v_2)$ if and only if
    \begin{align*}
        \big(\{v_1,v_2\} \in E(H),u_1=u_2\big) \textbf{ or } 
        \big(\{u_1,u_2\} \in E(G), v_1=v_2\big).
    \end{align*}
\end{defn}
\begin{lem}
    We have $\Dq^{G \sq H} = \Dq^G \otimes \Dq^H$, where $\otimes$ denotes the tensor product.
\end{lem}
\begin{proof}
   Let $(u_i, v_j) \in V(G \sq H)$, where $u_i \in V(G)$ and $v_j \in V(H)$. Since a path in $G \sq H$ consists of a combination of a path in $G$ and a path in $H$, we have 
    \begin{align*}
        \dist_{G \sq H}\big((u_i,v_k),(u_j,v_\ell)\big) = \dist_G(u_i,u_j) + \dist_H(v_k,v_\ell).
    \end{align*}
    This implies that \begin{align*}
        q^{\dist_{G \sq H}\big((u_i,v_k),(u_j,v_\ell)\big)} = q^{\dist_G(u_i,u_j)} q^{ \dist_H(v_k,v_\ell)}.
    \end{align*}
    If we examine entry-wise, this is the tensor product. Thus $\Dq^{G \sq H}=\Dq^G \otimes \Dq^H$. 
\end{proof}

Since the eigenvalues of the tensor of two matrices are found by taking all possible products of the eigenvalues of the original matrices, we have the following.

\begin{thm}\label{cart}
   We have $\spec(\Dq^{G \sq H})=\{\lambda \mu \, \mid \, \lambda \in \spec(\Dq^G), \mu \in \spec(\Dq^H)\}$.
\end{thm}

This is a fairly notable difference between the $\Dq$ matrix and the adjacency matrix. The adjacency matrix of the Cartesian product, $A^{G \sq H}$ is equivalent to $I_{|G|} \otimes A^H + A^G \otimes I_{|H|}$. The resulting matrix has eigenvalues that are the pairwise sums of the eigenvalues for as opposed to the pairwise products.

\begin{prop}\label{Prop:cutvertexcharpoly}
    If a graph $G$ has a cut vertex $u$ with a path of length two extending from it with $u \sim v_1$ and $v_1\sim v_2$ (see Figure~\ref{fig:cutvertexcharpoly}), then the characteristic polynomial satisfies \[P_{\Dq,G}(x) = \big((q^2+1)x-1 + q^2\big)P_{\Dq,G\setminus{v_2}}(x) - \big(q^2x^2\big)P_{\Dq,G\setminus \{v_1,v_2\}}(x).\] 
\end{prop}

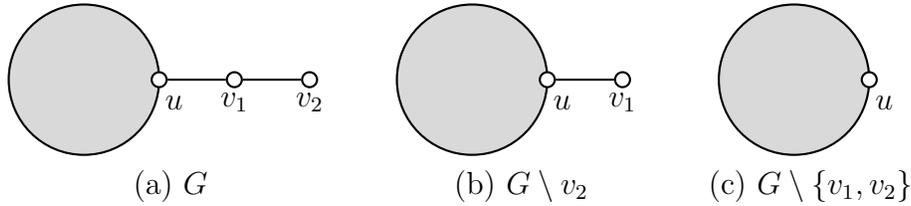
\begin{figure}[!htb]
\centering
\begin{tabular}{c@{\qquad}c@{\qquad}c}
    \begin{tikzpicture}
    \draw[black, thick,fill=black!15!white] (0,0) circle (1);
    \node[vertex] (v) at (1,0) {};
    \node at (1.2,-.3) {$u$};
    \node[vertex] (1) at (2,0) {};
    \node at (2,-.3) {$v_1$};
    \node[vertex] (0) at (3,0) {};
    \node at (3,-.3) {$v_2$};
    \draw[black, thick] (v) -- (1) -- (0);
    \end{tikzpicture}
&
    \begin{tikzpicture}
    \draw[black, thick,fill=black!15!white] (0,0) circle (1);
    \node[vertex] (v) at (1,0) {};
    \node at (1.2,-.3) {$u$};
    \node[vertex] (1) at (2,0) {};
    \node at (2,-.3) {$v_1$};
    \draw[black, thick] (v) -- (1);
    \end{tikzpicture}
&
    \begin{tikzpicture}
    \draw[black, thick,fill=black!15!white] (0,0) circle (1);
    \node[vertex] (v) at (1,0) {};
    \node at (1.2,-.3) {$u$};
    \end{tikzpicture}
\\
(a) $G$ & (b) $G\setminus v_2$ & (c) $G\setminus\{v_1,v_2\}$
\end{tabular}
\caption{The graphs for Proposition~\ref{Prop:cutvertexcharpoly}.}
\label{fig:cutvertexcharpoly}
\end{figure}

\begin{proof}
Let us start by partitioning our exponential distance matrix such that $v_1, v_2, u$ in that order are the first three columns and rows. Now, let us consider the determinant of $xI-\Dq$ (which is the characteristic polynomial) and perform row and column operations which do not affect the determinant (thus we do not affect our characteristic polynomial); namely, adding $-q$ times the second row/column to the first row/column. 

Since $u$ is a cut vertex, the shortest path from $v_1,v_2$ to the other vertices of the graph goes through $u$.  
 \begin{align*}
        \det(xI-\Dq) &= \det
        \left(\begin{array}{@{}ccc|ccc@{}}
    x-1 & -q & -q^2 & -q^2 \vec y \\
    -q & x-1 & -q & -q \vec y \\
    -q^2 & -q & x-1 &  -\vec y  \\\hline
    -q^2 \vec y^T & -q \vec y^T & -\vec y^T & M \\
        \end{array}\right) \\
        &= \det
        \left(\begin{array}{@{}ccc|ccc@{}}
    (q^2+1)x-1 & -qx & 0 & \mathbf{0}^T \\
    -qx & x-1 & -q & -q\vec y\\
    0 & -q & x-1 & -\vec y\\\hline
    \mathbf{0} & -q\vec y^T & -\vec y^T & M \\
        \end{array}\right).
\end{align*}
Using co-factor expansion along the first row, the result follows. 
\end{proof}

\begin{cor} \label{Cor:fkRecursion}
    If a graph $G$ has a cut vertex $u$ with a path of length $k$ extending from it with $u \sim v_1$, and graphs $H_1 = G\setminus \{v_2, \dots , v_{k}\}$, $H_2 = G\setminus \{v_1,\dots , v_{k}\}$ (See Figure~\ref{fig:fkRecursion}), then the characteristic polynomial satisfies 
    \[
    P_{\Dq,G}(x) = f_kP_{\Dq,H_1}(x) -q^2x^2f_{k-1}P_{\Dq,H_2}(x),
    \] 
    for $f_k$ with the following values.
    
    \[f_k = 
\begin{cases} 
      0 &  \text{if }k = 0, \\
      1 & \text{if } k = 1,\\
      [(q^2+1)x - 1 + q^2] &  \text{if }k = 2,\\
      f_2f_{k-1} - q^2x^2 f_{k-2} & \text{if }k > 2.
   \end{cases}
\]


\end{cor}

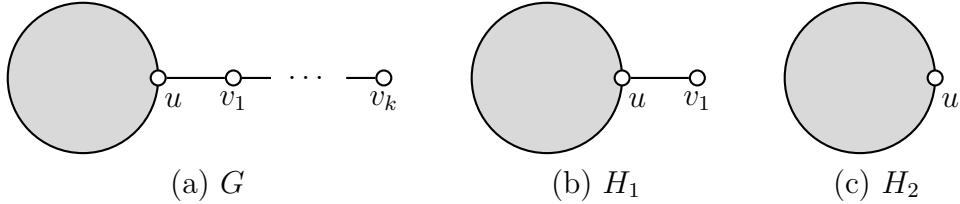
\begin{figure}[!htb]
\centering
\begin{tabular}{c@{\qquad}c@{\qquad}c}
    \begin{tikzpicture}
    \draw[black, thick,fill=black!15!white] (0,0) circle (1);
    \node[vertex] (v) at (1,0) {};
    \node at (1.2,-.3) {$u$};
    \node[vertex] (1) at (2,0) {};
    \node at (2,-.3) {$v_{1}$};
    \node[vertex] (0) at (4,0) {};
    \node at (4,-.3) {$v_k$};
    \draw[black, thick] (v) -- (1) -- (2.5,0) (3.5,0) -- (0);
    \node at (3,0) {\dots};
    \end{tikzpicture}
&
    \begin{tikzpicture}
    \draw[black, thick,fill=black!15!white] (0,0) circle (1);
    \node[vertex] (v) at (1,0) {};
    \node at (1.2,-.3) {$u$};
    \node[vertex] (1) at (2,0) {};
    \node at (2,-.3) {$v_{1}$};
    \draw[black, thick] (v) -- (1);
    \end{tikzpicture}
&
    \begin{tikzpicture}
    \draw[black, thick,fill=black!15!white] (0,0) circle (1);
    \node[vertex] (v) at (1,0) {};
    \node at (1.2,-.3) {$u$};
    \end{tikzpicture}
\\
(a) $G$ & (b) $H_1$ & (c) $H_2$
\end{tabular}
\caption{The graphs for Corollary~\ref{Cor:fkRecursion}.}
\label{fig:fkRecursion}
\end{figure}

\begin{proof}
    We proceed by induction. Let $G_k$ be the graph with a cut vertex from which extends a path of length $k$. The corollary is trivial for $G_0$ and $G_1$; for $G_2$ this is a restatement of Proposition~\ref{Prop:cutvertexcharpoly}. 
    
    Now, assume the corollary holds for all graphs up to $G_k$, and consider $G_{k+1}$. Using our induction hypotheses and Proposition~\ref{Prop:cutvertexcharpoly} we have, \begin{align*}
    P_{\Dq,G_{k+1}} &= f_2P_{\Dq,G_k} -q^2x^2P_{\Dq,G_{k-1}}\\
    &= \big(f_2f_kP_{\Dq,H_1} -q^2x^2f_2f_{k-1}P_{\Dq,H_2}\big) +\big(-q^2x^2f_{k-1}P_{\Dq,H_1} +(q^2x^2)^2f_{k-2}P_{\Dq,H_2}\big)\\
    &=\big(\underbrace{f_2f_k-q^2x^2f_{k-1}}_{=f_{k+1}}\big)P_{\Dq,H_1} -q^2x^2 \big(\underbrace{f_2f_{k-1}-q^2x^2f_{k-2}}_{=f_{k}}\big)P_{\Dq,H_2}.
    \qedhere
    \end{align*}
\end{proof}

We can use Corollary~\ref{Cor:fkRecursion} to find the characteristic polynomial of $\Dq$ for the graph $P_n$, the path on $n$ vertices.  We have the following recurrence with initial conditions
\begin{align*}
P_{\Dq,P_0}(x) &= 1, \\
P_{\Dq,P_1}(x) &= x-1,\text{ and}\\
P_{\Dq,P_n}(x) &= ((q^2+1)x-1+q^2)P_{\Dq,P_{n-1}}(x) -q^2x^2 P_{\Dq,P_{n-2}}(x).
\end{align*}
This is a second-order linear recurrence which can be readily solved by standard tools (or computer algebra system) and we get that $P_{\Dq,P_n}(x)=a (r_1)^n +b (r_2)^n$ where
\begin{align*}
\tau&=(q^2+1)x-1+q^2\\
r_1 &= \frac{\tau + \sqrt{\tau^2 - 4 q^2 x^2}}{2} \\
r_2 &= \frac{\tau - \sqrt{\tau^2 - 4 q^2 x^2}}{2} \\
a &= \frac{-1 + x - q^2 (1 + x) + \sqrt{(q^2-1) (  q^2 (1+x)^2-(x-1)^2)}}{2 \sqrt{\tau^2-4 q^2 x^2}}\\
b &= 2 q^4 (1 + x) + 2 (-1 + x)^2 (1 - x + \sqrt{(q^2-1) (q^2 (1 + x)^2-(x-1)^2)} \\
 & - 2 q^2 \frac{2 - 2 x + x^2 - x^3 + \sqrt{(q^2-1) (q^2 (1 + x)^2 - ( x -1)^2)}}{\sqrt{\tau^2-4 q^2 x^2} (\tau -     \sqrt{( q^2-1) (q^2 (1 + x)^2-( x -1)^2)}\,)^2}.
\end{align*}

Instead of appending a path to a graph, we can also append a star, as given in the following proposition.

\begin{prop} \label{Prop:leaves}
A graph $G$ with cut vertex $u$ that has $k$ leaves adjacent to it has a characteristic polynomial that satisfies \[P_{\Dq,G}(x)=k\big(x-(1-q^2) \big)^{k-1}P_{\Dq, H_1}(x)-(k-1)k\big(x-(1-q^2) \big)^k P_{\Dq, H_2}(x)\] where $H_1,H_2$ are the graphs as shown in Figure~\ref{fig:fkRecursion}.

\begin{figure}[!h]
    \centering
   
    \label{fig:k twins}
    \begin{tikzpicture}
\draw[thick, fill = black!15!white] (0,0) circle (1);
\node[vertex](v) at (0,1) {};
\node[vertex](u0) at (-1,2) {};
\node[vertex](u1) at (-.3,2) {};
\node at (0.4,2) {\dots};
\node[vertex](uk) at (1,2) {};
\draw[thick] (v) -- (u0) (v) -- (u1) (v) -- (uk);
\node at (0.35,1.1) {$u$};
\node at (-1,2.3) {$v_1$};
\node at (-.3,2.3) {$v_2$};
\node at (1,2.3) {$v_k$};
\end{tikzpicture}
\caption{A graph $G$ with $k$ leaves extending from cut vertex $u$.}
\end{figure}
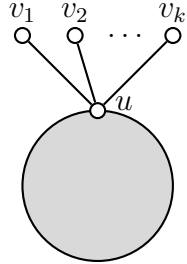

\end{prop}


\begin{proof}
We proceed by induction. Let $G_k$ be a graph with $k$ leaves extending from a single vertex. For $G_0$ the characteristic polynomial is $ 0 \cdot P_{\Dq,H_1}(x) + P_{\Dq,H_2}(x)$, and for $G_1$ the characteristic polynomial is $P_{\Dq,H_1}(x) + 0 P_{\Dq,H_2}(x)$. 

The characteristic polynomial of the $\Dq$ matrix for $G_k$ with $v_1, \ldots, v_k$ being the first $k$ rows is 
\begin{align*}
       \det(xI - \Dq) &=
        \det\left(\begin{array}{@{}cccc|ccc@{}}
    x-1 & -q^2 & \cdots & -q^2 &\vec y^T \\
    -q^2 & \ddots & \ddots & \vdots &  \vdots \\
    \vdots & \ddots & \ddots & -q^2 & \vec y^T \\
    -q^2 & \cdots & -q^2 & x-1 & \vec y ^T \\\hline
    \vec y & \cdots& \vec y & \vec y &  M\\
        \end{array}\right)\\
              &=
        \det\left(\begin{array}{@{}ccccc|ccc@{}}
    2(x-1)+2q^2 & -q^2-(x-1) & 0 &\cdots & 0 &  
    \mathbf{0}^T\\
    -q^2-(x-1) & x-1 & -q^2 & \cdots & -q^2 & \vec y^T  \\
    0  & -q^2 & \ddots & \ddots & \vdots &  \vdots \\
    \vdots & \vdots & \ddots & \ddots & -q^2 & \vec y^T  \\
    0 & -q^2  & \cdots & -q^2 & x-1 &  \vec y^T  \\\hline
    \mathbf{0} & \vec y & \cdots & \vec y & \vec y & M  \\
        \end{array}\right).
\end{align*}
%
%
%
%
Going from the first to the second line we use row operations which do not affect the determinant (and characteristic polynomial); namely, we add $(-1)$ times the second row/column to the first row/column. 

Now, assume that the proposition holds for all graphs up to $G_k$. Consider $G_{k+1}$. By co-factor expansion along the first row and the inductive hypothesis we have the following:
\begin{align*}
P_{\Dq,G_{k+1}}(x) =& 2\left( x-(1-q^2) \right)P_{\Dq,G_k} - (x-(1-q^2))^2 P_{\Dq,G_{k-1}}\\
=&2k\big(x-(1-q^2) \big)^{k}P_{\Dq, H_1}(x)- 2(k-1)k\big(x-(1-q^2) \big)^{k+1} P_{\Dq, H_2}(x)\\
&- (k-1)\big(x-(1-q^2) \big)^{k}P_{\Dq, H_1}(x) +(k-2)(k-1)\big(x-(1-q^2) \big)^{k+1} P_{\Dq, H_2}(x)\\
=& (k+1)\big(x-(1-q^2) \big)^{k}P_{\Dq, H_1}(x) -(k)(k+1)\big(x-(1-q^2) \big)^{k+1} P_{\Dq,
H_2}(x).\qedhere
\end{align*}

Note that if we add all possible edges amongst the $k$ leaves in the preceding result (forming a clique glued to a cut vertex), then the characteristic polynomial becomes 
\[
P_{\Dq,G}(x)=k(x-(1-q))^{k-1}P_{\Dq,H_1}(x) - (k-1)k(x-(1-q))^kP_{\Dq,H_2}(x),
\]
by the same argument with the only change being that the $q^2$ terms amongst the leaves in $\Dq$ become $q$.

These are both special cases of twin vertices.


\end{proof}

\begin{figure}[!h]
    \centering
    \begin{tikzpicture}
\draw[black, thick, fill=black!15!white](0,0) circle (1);
\node[vertex] (t1) at (-1,2) {};
\node at (1,2.5) {$t_2$};
\node at (-1,2.5) {$t_1$};
\node[vertex] (t2) at (1,2) {};
\draw[black, thick] (t1) -- (-.5,.5);
\draw[black, thick] (t1) -- (0,.5);
\draw[black, thick] (t1) -- (.5,.5);
\draw[black, thick] (t2) -- (-.5,.5);
\draw[black, thick] (t2) -- (0,.5);
\draw[black, thick] (t2) -- (.5,.5);
\draw[dashed, black, thick] (t1) -- (t2);
\end{tikzpicture}
\caption{Twin vertices.}
\label{fig:twins}
\end{figure}
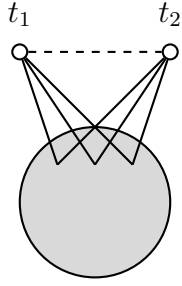

\begin{defn}
    Two vertices, $v_1,v_2$, in a graph $G$ are called \emph{twin vertices} if they are adjacent to the same set of vertices in $V(G \setminus \{v_1,v_2\})$, e.g.\ Figure~\ref{fig:twins}. If $v_1 \not \sim v_2$, we call them \emph{unconnected twins}. If $v_1 \sim v_2$, we call them \emph{connected twins}.
\end{defn}

\begin{prop}
    Let $G$ be a graph with a pair of twin vertices, $v_1,v_2$. Then its exponential distance matrix $\Dq$ has $(1-q^2)$ as an eigenvalue if $v_1 \not \sim v_2$ and $(1-q)$ as an eigenvalue if $v_1 \sim v_2$.
\end{prop}

\begin{proof}
    Let $G$ be a graph with disconnected twins $v_1,v_2$. Then if we let the first two rows and columns correspond to $v_1, v_2$, we can see that
    \begin{align*}
        \Dq = 
        \left(
            \begin{array}{c|c}
                 \begin{matrix}
                    1&q^2 \\
                    q^2 & 1
                 \end{matrix}
                 &  
                 \begin{matrix}
                    {\vec y}^T \\
                    {\vec y}^T
                 \end{matrix}
                 \\
                 \hline
                 \begin{matrix}
                    \vec y & \vec y
                 \end{matrix}
                 & M
            \end{array}
        \right).
    \end{align*}
    Now by computation,
    \begin{align*}
        \Dq 
        \begin{pmatrix*}[r]
            1\\
            -1\\
            \mathbf{0}
        \end{pmatrix*}
        = \begin{pmatrix}
            1-q^2\\
            q^2-1\\
            \mathbf{0}\\
        \end{pmatrix}
        = (1-q^2)
        \begin{pmatrix*}[r]
            1\\
            -1\\
            \mathbf{0}
        \end{pmatrix*}.
    \end{align*}
    Thus, $\vec x = \langle  1, -1, 0, \ldots,0 \rangle ^T$ is an eigenvector for $\Dq$ with eigenvalue $(1-q^2)$. We use a similar argument to show that $(1-q)$ is an eigenvalue if $v_1, v_2$ are connected twins. 
\end{proof}

\subsection{Connections to other matrices}

\begin{prop} \label{regular}
    If $G$ is a $k$-regular graph on $n$ vertices with $\diam (G) \leq 2$ and has adjacency matrix $A$ with spectrum $\{\mu_1=k, \mu_2,\dotsc, \mu_n\}$, then the exponential distance matrix has eigenvalues
    \begin{align*}
        \lambda_i = 
        \begin{cases}
            (q^2)n - (q^2-q)k -(q^2-1) & \text{for }i = 1,\\
            -(q^2-q)\mu_i -q^2 +1 \ & \text{for }2 \leq i \leq n.\\
        \end{cases}
    \end{align*}
\end{prop}

\begin{proof}
    Consider the $(i,j)$-th entry of $\Dq$
\[(\Dq)_{i,j} = 
\begin{cases} 
      1 &  \text{if }i=j, \\
      q & \text{if } v_i \sim v_j, \\
      q^2 &  \text{ otherwise}.
   \end{cases}
\]
Thus, we can express $\Dq$ in terms of $A$, $J$ (the all $1$s matrix), and $I$ as follows:
\begin{align*}
    \Dq = (q^2) J - (q^2-q) A - (q^2-1)I.
\end{align*}
Since $G$ is regular, $A$ has constant row sums and $(k,\mathbf{1})$ is an eigenpair. Thus,
\begin{align*}
    \Dq \mathbf{1} &= \big[(q^2) J - (q^2-q) A - (q^2-1)I\big] \mathbf{1} \\
    &=\big[q^2 n - (q^2-q)k - (q^2-1)\big]\mathbf{1}.
\end{align*}
Thus $(q^2 n - (q^2-q)k - (q^2-1),\mathbf{1})$ is an eigenpair of $\Dq$. Since $A$ is a Hermitian matrix, its remaining eigenvectors are orthogonal to $\mathbf{1}$. Thus, if $(\mu_i,\vec x)$ is one of the remaining eigenpairs of $A$,
\begin{align*}
    \Dq \vec x &= \big[(q^2) J - (q^2-q) A - (q^2-1)I\big] \vec x \\
    &= \big[-(q^2-q) \mu_i  - (q^2-1)\big] \vec x,\\
\end{align*}
gives the remaining eigenpairs of $\Dq$.
\end{proof}

\begin{defn}
    The \emph{join} of two graphs, $G,H$, denoted $G \vee H$, has $V(G \vee H)$ as the disjoint union of $V(G),V(H)$ and $E(G \vee H) = E(G) \cup E(H) \cup \{\{u,v\}\mid u \in V(G), v \in V(H)\}$. 
\end{defn}

\begin{prop}\label{join}
    Let $G$ be a graph on $n$ vertices that is $k$-regular whose adjacency matrix $A^G$ has spectrum $\{\lambda_1 = k,\dotsc,\lambda_n\}$. Let $H$ be a graph on $m$ vertices that is $\ell$-regular whose adjacency matrix $A^H$ has spectrum $\{\mu_1 = \ell,\dotsc,\mu_m\}$. Then
    \[\spec(\Dq^{G \vee H})=\{\phi_1, \phi_2\} \cup \{-(q^2-q)\lambda_i - q^2+1 \mid 2 \leq i \leq n\} \cup \{ -(q^2-q)\mu_j - q^2+1 \mid 2\leq j \leq m\},\] with 
    \begin{align*}
        \phi_1, \phi_2 = \pm\frac{q \sqrt{k^2 q^2 - 2 k^2 q + k^2 - 2 k \ell q^2 + 2 k \ell + \ell^2 q^2 + 2 \ell^2 q + \ell^2 + 4 m n}}{2}\\
        - \frac{-k q^2 + k q - \ell q^2 - \ell q + 2 n q^2 - 2 q^2 + 2}{2}.
    \end{align*}
\end{prop}
\begin{proof}
    Notice that $\diam(G\vee H) \leq 2$ because the edges added by the join operation guarantee a path of length less than or equal to 2 between any vertices. Thus,
    \[
         \Dq^{G \vee H} = 
         \left(\begin{array}{@{}c|c@{}}
            q^2 J + (q-q^2) A^G - (q^2-1)I & qJ\\
            \hline
            q J &
             q^2 J + (q-q^2) A^H - (q^2-1)I
        \end{array}\right),
     \]
     where $A^G, A^H$ are the adjacency matrices for $G,H$ respectively. Since $G,H$ are regular, $\mathbf{1}$ is an eigenvector for $A^G$ and $A^H$ and all the remaining eigenvectors for the two matrices are orthogonal to $\mathbf{1}$. 
     
     Since each block has a constant row sum, we can find eigenvectors of the form $\big({\alpha \mathbf1\atop\beta\mathbf1}\big)$, i.e.\ by doing an equitable partition.  This reduces to finding the eigenvectors and eigenvalues for the following $2 \times 2$ matrix. 
    \begin{align*}
        \begin{pmatrix}
            1+q k+ q^2 (n-k -1) & m q \\
            n q & 1+q \ell + q^2 (m-\ell -1) \\
        \end{pmatrix}
    \end{align*}
    This gives the eigenvalues $\phi_1, \phi_2$.
    
    For the remaining eigenvalues, consider the eigenpairs for $A^G$, $(\lambda_i, \vec x_i)$, where $\vec x_i$ is orthogonal to $\mathbf{1}$. 
    Similarly, consider the eigenpairs, $(\mu_j,\vec y_j)$, of $A^H$ where $\vec y_j$ is orthogonal to $\mathbf{1}$. 
     Construct the following vectors
     \begin{align*}
        \begin{pmatrix}
            \vec x_i \\
            \hline 
            \mathbf{0}\\
        \end{pmatrix} \text{, for } 2 \leq i \leq n \text{ and }
         \begin{pmatrix}
            \mathbf{0}\\
            \hline 
            \vec y_j \\
        \end{pmatrix} \text{, for } 2 \leq j \leq m.
    \end{align*}
    Notice that
    \begin{align*}
            \Dq^{G \vee H}
            \begin{pmatrix}
                \vec x_i \\
                \hline 
                \mathbf{0}\\
            \end{pmatrix} 
            &= 
             \left(\begin{array}{@{}c|c@{}}
                q^2 J + (q-q^2) A^G - (q^2-1)I & qJ\\
                \hline
                q J &
                 q^2 J + (q-q^2) A^H - (q^2-1)I
            \end{array}\right)
            \begin{pmatrix}
                \vec x_i \\
                \hline 
                \mathbf{0}\\
            \end{pmatrix}\\
            &= 
            \begin{pmatrix}
                (q^2 J + (q-q^2) A^G - (q^2-1)I)\vec x_i \\
                \hline 
                qJ \vec x_i
            \end{pmatrix}\\
            &= \big((q-q^2) \lambda_i- q^2 + 1\big)
             \begin{pmatrix}
                \vec x_i \\
                \hline 
                \mathbf{0}\\
            \end{pmatrix}.
        \end{align*}
    Thus, for $2 \leq i \leq n$, we have $(q-q^2) \lambda_i- q^2 + 1$ is  an eigenvalue for $\Dq^{G \vee H}$. A similar computation shows $(q-q^2) \mu_j - q^2 + 1$ is also an eigenvalue for $2 \leq j \leq m$.
\end{proof}
   
The following allows us to determine the spectrum of the distance matrix for graphs with diameter at most two from the spectrum of the exponential distance matrix.

\begin{prop}\label{prop:distpoly}
    Let $G$ be graph on $n$ vertices such that $\diam (G) \leq 2$. Let $\mathcal{D}_2^G$ be the exponential distance matrix of $G$ with $q=2$. Then
    \[
        P_{\mathcal{D},G}(x) = \frac{1}{2^n} P_{\mathcal{D}_2,G}(2x+1).
    \]
\end{prop}
\begin{proof}
    Consider the $(i,j)$-th entry of $\mathcal{D}_2$. Since $G$ has diameter less than or equal to 2,
    \[(\mathcal{D}_2)_{i,j} = 
        \begin{cases} 
        1 & \text{if }i = j,\\
        2 & \text{if }v_i \sim v_j ,\\
        4 & \text{otherwise}.\\
        \end{cases}
    \]
    Now consider the $(i,j)$-th entry of the standard distance matrix, $\mathcal{D}$,
    \[(\mathcal{D})_{i,j} = 
        \begin{cases} 
        0 & \text{if }i = j, \\
        1 & \text{if }v_i \sim v_j, \\
        2 & \text{otherwise.}
        \end{cases}
    \]
    Thus, $ \mathcal{D} = \frac{1}{2}(\mathcal{D}_2 - I)$, so
    \begin{align*}
        P_{\mathcal{D},G}(x)  &= \det(xI - \mathcal{D})\\
        &= \det(xI - (\frac{1}{2}(\mathcal{D}_2 - I)))\\
        &= \det(\frac{1}{2}((2x+1)I -\mathcal{D}_2)) \\
        &= \frac{1}{2^n} \det((2x+1)I -\mathcal{D}_2)\\
        &= \frac{1}{2^n} P_{\mathcal{D}_2,G}(2x+1).\qedhere
    \end{align*}
\end{proof}

\begin{cor}
    If $G,H$ are graphs that are $\Dq$-cospectral and both have diameter less than or equal to 2, then $G,H$ are also $\mathcal{D}$-cospectral.
\end{cor}

Now, we prove the exponential distance matrix's characteristic polynomial stores the adjacency matrix's characteristic polynomial. 

\begin{thm}\label{adjacency}
    If $a_k$ is the coefficient of $q^k x^{n-k}$ in $P_{\Dq,G}(x+1)$, then $P_{A,G}(x) = a_n x^n + a_{n-1} x^{n-1} + \cdots + a_1 x + a_0$.
\end{thm}

\begin{proof}
    Notice
    \begin{align*}
        P_{\Dq,G}(x+1) &= \det((x+1)I - \Dq)\\
        &= \sum_{\sigma \in {S_n}}\Big(\sign(\sigma)\prod_{i=1}^n ((x+1)I-\Dq)_{i,\sigma(i)}\Big),
    \end{align*}
    where the $(i,j)$-th entry of $((x+1)I - \Dq)$ is $x$ if $i=j$, and $-q^{\dist(i,j)}$ otherwise. Thus, in order to get a term with $x^{n-k}$, there must be $(n-k)$ fixed points in our permutation $\sigma$. To get a $q^k x^{n-k}$ term, the remaining $k$ non-fixed points must each be $-q$ in our product. This corresponds to a selection of entries $(i,j)$ such that $v_i \sim v_j$. Now notice
    \begin{align*}
       P_{A,G}(x) &= \det(xI - A),
    \end{align*}
    where the $(i,j)$-th entry of $(xI - A)$ is $x$ if $i=j$, and $-A_{i,j}$ otherwise. To get a term with $x^{n-k}$, there must be $(n-k)$ fixed points and the remaining $k$ terms must all be $-1$. This also corresponds to a selection of entries $(i,j)$ such that $v_i \sim v_j$. Thus, the contribution from the permutations (including the sign term) are consistent for the two matrices. So the coefficient of $q^k x^{n-k}$ in $P_{\Dq,G}(x+1)$ is the coefficient of $x^{n-k}$ in $P_{A,G}(x)$.
\end{proof}

This theorem shows that the characteristic polynomial of the adjacency is encoded in the exponential distance matrix. Thus, any information that is preserved by the spectrum of the adjacency matrix is also preserved by the $\Dq$-spectrum. 

The complement of a graph $G$, denoted $\overline{G}$, has the same vertex set as $G$ so that $u \sim v$ in $G$ if and only if $u\not\sim v$ in $\overline{G}$. 

\begin{prop}\label{acomp}
    Let $\diam (G) \leq 2$. If $a_k$ is the coefficient of $q^{2k} x^{n-k}$ in $P_{\Dq,G}(x+1)$, then $P_{A,\overline{G}}(x) = a_n x^n + a_{n-1} x^{n-1} + \cdots + a_1 x + a_0$.
\end{prop}
\begin{proof}
    The statement can be proven in a manner very similar to the proof of Theorem~\ref{adjacency}. The key observation is that for $i\ne j$, $(\Dq^G)_{i,j} = q^2$ if and only if $(\Dq^{\overline{G}})_{i,j} = q$. This is because $\{v_i,v_j\} \not \in E(G)$ if and only if $\{v_i,v_j\}\in E(\overline{G)}$, by the definition of the complement.
\end{proof}

\subsection{Spectra of certain graphs}
Using the tools from the previous sections, we can now explicitly determine the spectra of certain families of graphs. 

\begin{prop}\label{kspecprop}
    The exponential matrix, $\Dq$, of a complete graph on $n$ vertices has eigenvalues $q(n-1)+1$ with multiplicity one and $(1-q)$ with multiplicity $n-1$.
\end{prop}
\begin{proof}
    Notice that $K_n$ is a $(n-1)$-regular graph with diameter $1$. It is also readily known that the eigenvalues for the adjacency matrix of $K_n$ are $(n-1)$ with multiplicity one and $-1$ with multiplicity $n-1$. Thus, we can apply Proposition~\ref{regular} to see that the $\Dq$ matrix of $K_n$ has eigenvalues $(q^2)n - (q^2-q)(n-1) -(q^2-1) =  q(n-1)+1$ with multiplicity one and $(q^2-q)(-1) -q^2 +1  = 1-q$ with multiplicity $n-1$.
\end{proof}
\begin{prop}
    The exponential distance matrix, of the hypercube on $2^n$ vertices, $Q_n$, has eigenvalues $(1-q)^k(1+q)^{n-k}$ with multiplicity of $\binom{n}{k}$, for $0\leq k \leq n$.
\end{prop}
\begin{proof}
    First, notice that $Q_n = P_2 \sq P_2 \sq \cdots \sq P_2$, i.e., the hypercube is the Cartesian product of $n$ copies of the path on two vertices. The exponential distance matrix of $P_2$ is
    \begin{align*}
        \Dq^{P_2} = 
        \begin{pmatrix}
            1 & q \\
            q & 1 \\
        \end{pmatrix},
    \end{align*}
    which has spectrum $\{(1-q),(1+q)\}$. By Theorem~\ref{cart}, we know that all the eigenvalues for $P_2 \sq P_2 \sq \cdots \sq P_2$ can be created by picking either $(1-q)$ or $(1+q)$ for each of the $n$ copies of $P_2$ and taking the resulting product. For $0\leq k \leq n$, we can select $k$ copies of $(1-q)$, which forces the remaining $(n-k)$ selections to be $(1+q)$. There are $\binom{n}{k}$ ways to do this, so this completes the proof.
\end{proof}

\begin{prop}
The exponential distance matrix, $\Dq^{C_n}$, of the cycle on $n$ vertices, $C_n$, has the spectrum \[
\bigg\{1+2\sum_{k=1}^{(n-1)/2}q^k\cos(\tfrac{2\pi kj}n) \mid 1\le j \le n\bigg\}
\]
when $n$ is odd and
\[
\bigg\{1+(-1)^jq^{n/2}+2\sum_{k=1}^{(n-2)/2}q^k\cos(\tfrac{2\pi kj}n) \mid 1\le j \le n\bigg\}
\]
when $n$ is even.
%
\end{prop}
\begin{proof}
   For any circulant matrix $M$, $\spec (M) = \{m_1 + m_2\zeta_i + m_3\zeta_i^2 + \dots + m_n\zeta_i^{n-1} \mid 1 \leq i \leq n\}$, where $m_1, \dots, m_n$ are the entries in the first row of $M$. The matrix $\Dq^{C_n}$ is circulant and has entries in its first row $1, q, \ldots, q^{\frac{n}{2}}, q^{\frac{n}{2}-1}, \ldots, q$ for even $n$ and $1, q, \ldots, q^{\frac{n-1}{2}}, q^{\frac{n-1}{2}}, \ldots, q$ for odd $n$. The result emerges immediately from these two facts, along with  $2\cos({\theta}) = e^{i\theta} + e^{-i\theta}$
\end{proof}
We can also generate $\Dq$ the spectrum of the wheel $W_n$ on $n+1$ vertices using Proposition~\ref{join}.
\begin{prop}
The spectrum of $\Dq^{W_n}$ on $n+1$ vertices is equal to $$\{\phi_1, \phi_2\} \cup\{-(q^2-q)(2\cos(\frac{2\pi j}{n})) - q^2 + 1 \mid 2 \leq j \leq n\},$$ with values of $\phi_1, \phi_2$ established in Proposition~\ref{join}, for $k=2$, $\ell = 0$, and $m = 1$.
\end{prop}
\begin{proof}
    Because $W_n = C_n \vee K_1$, we can produce the result stated in the proposition by applying the above parameters to the formula established in Proposition~\ref{join}.
\end{proof}

One nice family of graphs are the Kneser graphs $KG(n,r)$ which has as its ${n\choose r}$ vertices all of the $r$ element subsets of $[n]=\{1,2,\ldots,n\}$ and two vertices are connected by an edge if and only if the corresponding sets are disjoint.  As an example $KG(n,1)$ is the complete graph on $n$ vertices; and $KG(5,2)$ is the Petersen graph.  The spectral properties of these graphs have been well studied and so we are able to find the spectrum for these graphs using an approach adapted from \cite{Aalipour}.  (In general, the spectrum for the Kneser graph can be determined for any matrix where the entries are based on a function of the distance.)

\begin{prop}
The $\Dq$ spectrum of the Kneser Graph $KG(n,r)$ consists of the eigenvalues
\begin{align*}
    \lambda_j = \sum_{i=0} ^r q^{f(i)}p_i(j)
\end{align*}
with multiplicity $m_j = \frac{n-2j+1}{n-j+1}\binom{n}{j}$ for $j = 0,1, \dots, r$, where \begin{align*}
    p_i(j) = \sum_{t=0}^j (-1)^t \binom{j}{t}\binom{r-j}{i-t} \binom{n-r-j}{i-t}
\end{align*}
and 
\begin{align*}
    f(i) = \min\Big\{2\Big\lceil\frac{i}{n-2r} \Big\rceil, 2\Big\lceil\frac{r-i}{n-2r} \Big\rceil +1 \Big\}.
\end{align*}

\end{prop}
\begin{proof}
Let $A_i$ be the adjacency matrix of the generalized Kneser graph $J(n;r;i)$ for $0 \leq i \leq r$ where in $J(n;r;i)$ are adjacent if and only if their corresponding subsets share $r-i$ elements.  It is known \cite{Aalipour,BI} that the matrices $A_i$ form a commutative following (in particular share common eigenvectors) and moreover the spectra of each $A_i$ is the set $p_i(j)$ with multiplicity $m_j$.  Note that the $A_i$ form a decomposition of the all $1$s matrix.

For the distance matrix it was shown \cite{Aalipour}
\begin{align*}
    \mathcal{D}=\sum_{i=0}^r f(i)A_i
\end{align*}
where 
$f(i) = \min\Big\{2\Big\lceil\frac{i}{n-2r} \Big\rceil, 2\Big\lceil\frac{r-i}{n-2r} \Big\rceil +1 \Big\}$.  In particular, it follows that if an entry in $A_i$ is $1$ the distance between the corresponding vertices is $f(i)$.  So for the exponential distance matrix the corresponding entry will be $q^{f(i)}$, so we have
\begin{align*}
    \mathcal{D}_q=\sum_{i=0}^r q^{f(i)}A_i.
\end{align*}
The result for the spectrum now follows.
\end{proof}

\section{Information from the characteristic polynomial}\label{sec:poly}
In this section we examine what information can be obtained about a graph from its $\Dq$ spectrum and/or characteristic polynomial. These observations tell us which properties of the graph are preserved by the eigenvalues of the $\Dq$ matrix and what restrictions can be placed on a pair of graphs if they are $\Dq$-cospectral. 

\begin{lem} \label{lem:comp}
Let $\mathcal{D}_1$ be $\Dq$ with $q=1$. We have $\spec(\mathcal{D}_1) =\{t_1, t_2, \dots,t_k, 0, \dots, 0 \}$
    where $k$ is the number of components of $G$ and $t_i$ is the size of the $i$-th component of $G$.
\end{lem}
\begin{proof} The exponential distance matrix of $G$ can be written in block diagonal form such that
    \begin{align*}
        \Dq^G = 
        \begin{pmatrix}
            \Dq^{G_1} & O & \cdots & O \\
            O & \Dq^{G_2} & \ddots & \vdots\\
            \vdots & \ddots & \ddots & O\\
            O & \cdots & O & \Dq^{G_k} \\
        \end{pmatrix},
    \end{align*}
    where each $\Dq^{G_i}$ is the $\Dq$ matrix of $G_i$, the $i$-th connected component of $G$.  We know $\spec(\Dq) = \bigcup_{i=1}^k \spec(\Dq^{G_i})$, so let us consider $\spec(\Dq^{G_i})$. Notice that $\mathcal{D}_1^{G_i} = J_{t_i}$
    where $J_{t_i}$ is a ${t_i} \times {t_i}$ matrix of all ones. We know the $\spec(J_{t_i}) = \{t_i, 0^{(t_i - 1)}\}$. We apply this to each component of $G$ to complete the proof. 
\end{proof}
\begin{obs}\label{component}
    The preceding proof allows us to derive the following result. Let $G$ be a graph with exponential distance matrix $\Dq^G$ with eigenvalue $\lambda_i(q)$, such that $\displaystyle{\lim_{q \to 1}\lambda_i(q)} = t_i \not = 0$. Then there exists a component, $G_i$, of $G$ that has $t_i$ vertices with eigenvalue $\lambda_i(q)$.
\end{obs}
\begin{lem}\label{distancelemma}
    Given the characteristic polynomial of a graph $G$, one can determine the number of pairs of vertices that are distance $k$ apart for any $k$. 
\end{lem}
\begin{proof}
    The coefficient of $x^{n-2}$ in the characteristic polynomial for $P_{\Dq,G}(x+1)$, as considered in Theorem~\ref{adjacency}, is
    \[
        -\sum_{i\ne j}(\Dq^G)_{i,j}(\Dq^G)_{j,i}=
        -\sum_{i\ne j}q^{2\dist(i,j)}=
        -\sum_k b_kq^{2k},
    \]
    where $b_k$ is the number of pairs of vertices at distance $k$.
    %
    %
\end{proof}

\begin{thm}\label{specinfo}
    Given the spectrum of the exponential distance matrix of a graph $G$ we can determine the following properties:
    \begin{enumerate}
        \item The number of components.
        \item The size of each component.
        \item The number of pairs of vertices distance $k$ apart for any $k$.
        \item The number of edges.
        \item The diameter.
    \end{enumerate}
\end{thm}

\begin{proof}
    This follows from Observation~\ref{component} and Lemma~\ref{distancelemma}.
\end{proof}

\begin{cor}
Whether or not a graph is a forest can be determined by the spectrum of $\Dq$. 
\end{cor}

\begin{prop}\label{P3}
    Let $P_{\Dq,G}(x)$ be the characteristic polynomial of the exponential distance matrix of a graph $G$ on $n$ vertices. If $k$ is the coefficient of $q^4 x^{n-3}$ in $P_{\Dq,G}(x+1)$, then the number of induced copies of $P_3$ in G is $-\frac{k}{2}$. 
\end{prop}

\begin{proof}
    Consider the coefficient of $q^4 x^{n-3}$ in $P_{\Dq,G}(x+1)$, call it $k$. 
    We can again use our technique from the proof of Theorem~\ref{adjacency}. To get a $q^4 x^{n-3}$ term, the non-fixed points, say $(i,j), (j,k), (k,i)$, must multiply together to $-q^4$. This is only possible if two of the terms are $-q$ and one is $-q^2$. This set of distances is only possible with an induced $P_3$.

    Furthermore, each $P_3$ contributes two terms of $-q^4 x^{n-3}$ because the matrix is symmetric so the product of $(i,j), (j,k), (k,i)$ will be the same as the product of $(j,i), (k, j), (i,k)$.
\end{proof}

Notice that the number of induced $P_3$'s is distinct from the number of pairs of vertices that are distance two apart. Consider $C_4$. It has four induced copies of $P_3$, but only two pairs of vertices that are distance two apart. Furthermore, this is a property that is determined for the spectrum of $\Dq$, but not for the spectrum of $A$. The graphs in the Saltire Pair (see Figure~\ref{fig: Saltire pair}) are $A$-cospectral, but they not have the same number of induced copies of $P_3$.
\begin{figure}[h]
    \centering
    \begin{tikzpicture}[scale=0.707]
\node[vertex] (a) at (0,0) {};
\node[vertex] (b) at (2,0) {};
\node[vertex] (c) at (0,2) {};
\node[vertex] (d) at (2,2) {};
\node[vertex] (e) at (1,1) {};
\draw[thick] (a)--(e) (b)--(e) (c)--(e) (d)--(e);
\end{tikzpicture}
\hspace{20pt}
\begin{tikzpicture}[scale=0.707]
\node[vertex] (a) at (0,0) {};
\node[vertex] (b) at (2,0) {};
\node[vertex] (c) at (0,2) {};
\node[vertex] (d) at (2,2) {};
\node[vertex] (e) at (1,1) {};
\draw[thick] (a)--(b)--(d)--(c)--(a);

\end{tikzpicture}
    \caption{Saltire pair.}
    \label{fig: Saltire pair}

\end{figure}
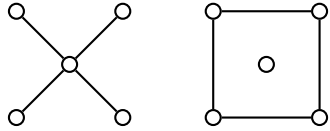
\subsection{Graphs determined by their spectra}
We will now show families of graphs that are spectrally determined for the exponential distance matrix. 

\begin{prop}
    If a graph $G$ is spectrally determined for the adjacency matrix, it is also spectrally determined for $\Dq$.
\end{prop}
\begin{proof}
    This follows from Theorem~\ref{adjacency}.
\end{proof}

There are many known families that are known to be spectrally determined for $A$ (see \cite{vanDam}), these will also be spectrally determined for $\Dq$.

\begin{prop}
\begin{sloppypar}
    If a graph $G$ is spectrally determined for the adjacency matrix and ${\diam(\overline{G})\leq 2}$, then $\overline{G}$ is spectrally determined for $\Dq$.
\end{sloppypar}
\end{prop}
\begin{proof}
    This follows from Proposition~\ref{acomp}
\end{proof}
\begin{thm}\label{complete}
    The exponential distance matrix of a graph $G$ has an eigenvalue in the form $(t-1)q+1$ (where $t$ is an integer) if and only if $G$ has $K_t$ as one of its components.
\end{thm}
\begin{proof}
    The reverse direction follows easily from Proposition~\ref{union} and Proposition~\ref{kspecprop}. The exponential distance matrix of $K_t$ has eigenvalue $(t-1)q+1$, so if $K_t$ is a component of $G$, the exponential distance matrix of $G$ will also have $(t-1)q+1$ as an eigenvalue.

    For the forward direction, notice that if $\lambda = (t-1)q+1$ is an eigenvalue of $\Dq^G$, $\lim_{q \to 1} \lambda = \lim_{q \to 1} (t-1)q+1 = t$. By Observation~\ref{component}, $G$ has a component with $t$ vertices, call it $G_j$. Furthermore, the exponential distance matrix of $G_j$, $\Dq^{G_j}$ has $(t-1)q+1$ as an eigenvalue.  We now can restrict ourselves to this component.
    
    For some $q$ with $0<q<1$ we have that $(t-1)q+1$ is the largest eigenvalue in size (since all other eigenvalues $\to 0$ as $q\to1$).  Now fix such a value of $q$.  By Perron-Frobenius we may asssume that the eigenvector for the eigenvalue corresponding to $(t-1)q+1$ has the form $\vec x = \langle x_1, \dotsc, x_t\rangle^T$ where $x_1=1$ is the largest entry and all other entries are positive.

    We have
    \begin{align*}
        \Dq^{G_j}
        \begin{pmatrix}
            x_1 \\
            \vdots \\
            x_t\\
        \end{pmatrix}
        = ((t-1)q+1)
        \begin{pmatrix}
            x_1 \\
            \vdots \\
            x_t\\
        \end{pmatrix}.
    \end{align*}
    Looking at the first coordinate on both sides we see
    \begin{align*}
        ((t-1)q+1) &= 1+q\sum_{\substack{x_i\\ \dist(x_i,x_1) = 1}}x_i+q^2\sum_{\substack{x_i\\ \dist(x_i,x_1) = 2}}x_i + \cdots \\
        &\leq 1+q\sum_{\substack{x_i\\ 2\leq i \leq t}}x_i\\
        &\leq (t-1)q+1.
    \end{align*}
    The first inequality from our assumptions that $0<q<1$ and the second inequality from our assumption that each $x_i$ is at most $1$. For this to hold, our two inequalities must be equalities.  This is only possible if the distance from $1$ to any other vertex is $1$ (so $1$ is adjacent to all vertices) \emph{and} if $\vec x=\mathbf1$.
    
    Since $\vec{x}=\mathbf{1}$ we can now repeat this same argument for any vertex in the component and conclude that any pairs of vertices are adjacent.  This shows that the component is a clique, as desired.
    %
\end{proof}

\begin{prop}\label{pathprop}
    Given the spectrum of the exponential distance matrix of a graph $G$, one can determine whether or not $G$ is a union of paths.
\end{prop}

\begin{proof}
    Let $f(G,k)$ be defined as the number of pairs of vertices that are distance $k$ apart in a graph $G$. Notice that by Theorem~\ref{specinfo}, we can determine the value of $f(G,k)$ for any $k$ as well as the number of components of $G$ and the size of each component. Let $G_1$ be the largest component, say it has size $t_1$. Notice that $t_1 \geq \diam(G)+1$. If $t_1 > \diam(G)+1$, $G_1$ cannot be a path because there is no pair of vertices that are distance $(t_1 - 1)$ apart. If $t_1 = \diam(G)+1$, $G_1$ can be assumed to be a path because the path is the only graph on $t_1$ vertices with diameter $t_1-1$. 
    
    If $G_1$ is a path, we can remove it from the graph. Notice that for a path on $t_1$ vertices, $P_{t_1}$, $f(P_{t_1},k) = t_1-k$. Thus, $f(G\setminus G_1, k) =  f(G, k) - (t_1 - k)$. Now we can repeat the process by comparing the size of the next largest component $G_2$ with $\diam(G \setminus G_1)$. We can do this for each component of $G$ which will then determine whether or not $G$ is a union of paths.
\end{proof}

We can extend the previous result for unions of complete and path graphs. 

\begin{thm}
    Given the spectrum of the exponential distance matrix of a graph $G$, one can determine whether or not $G$ is a union of paths and complete graphs.
\end{thm}

\begin{proof}
    By Theorem~\ref{complete}, we can determine the number of components that are cliques and the size of each clique. We can proceed by removing one clique at a time. If one of the components is $K_t$, we know removing it will remove $\frac{t(t-1)}{2}$ edges, and will not change the number of pairs of vertices that have distance greater than one between them. Thus, we know 
    \begin{align*}
        f(G\setminus K_t,k) = 
        \begin{cases}
          f(G,1) - \frac{t(t-1)}{2} & \text{if } k= 1,\\
          f(G,k) & \text{otherwise}.
        \end{cases}
    \end{align*}
    We can repeat this process for each component that is a clique. Once we have removed all the cliques, we are left with a subgraph $G^*$ for which we know $f(G^*,k)$, for each $k$ as well as the size and number of components. Thus, using Proposition~\ref{pathprop}, we can determine whether the remaining components are paths or not. 
\end{proof}

\section{Cospectrality}\label{sec:cospec}

A pair of graphs that have the same spectrum are called \emph{cospectral} graphs. In this section we will discuss non-isomorphic graphs that are cospectral for the exponential distance matrix for all values of $q$. By Theorem~\ref{adjacency}, any graphs that are cospectral for $\Dq$ are automatically cospectral for the adjacency matrix as well. Note that the converse does not hold; recall that the Saltire Pair, shown in Figure~\ref{fig: Saltire pair}, is $A$-cospectral, but not $\Dq$-cospectral. 

Cospectrality is an interesting question because it allows us to see the weaknesses of our matrix (what the spectrum does not preserve). We will see one example where the exponential distance matrix does not preserve the largest degree for all values of $q$ (see Figure~\ref{fig:Dq not D}). 

\begin{defn}
    Let $G$ and $H$ be graphs such that $u\in V(G)$ and $v \in V(H)$. Then the graph that is the result of \emph{vertex identification} of $G$ and $H$ with vertices $u,v$, denoted $G \vxi{u}{v} H$, is constructed by identifying the vertices $u$ and $v$ into a single vertex in the graph $G\cup H$. 
\end{defn}

\begin{thm} \label{glue}
    Let $G_1$ and $G_2$ be graphs with vertices $u_1 \in V(G_1)$ and $u_2 \in V(G_2)$. Construct $G_1^*$ by attaching a leaf to vertex $u_1$ and construct $G_2^*$ by attaching a leaf to vertex $u_2$, as shown in Figure~\ref{fig:G*}. Let $H$ be a graph with vertex $v \in V(H)$. If $G_1$ and $G_2$ are $\Dq$-cospectral and $G_1^*$ and $G_2^*$ are $\Dq$-cospectral, then the vertex identifications $G_1 \vxi{u_1}{v} H$ and $G_2 \vxi{u_2}{v} H$, shown in Figure~\ref{fig: vertex id}, are also $\Dq$-cospectral.
\end{thm}

\begin{figure}[!h]
    \centering
\begin{tikzpicture}
\draw[black, thick, fill = black!15!white] (-1,0) circle (1);
\node[vertex] (v) at (0,0) {};
\node at (.2,-.3) {$u_1$};
\node[vertex] (a1) at (1,0) {};
\draw[black, thick](v) -- (a1);
\node at (-1,0) {$G_1$};
\end{tikzpicture}
\hspace{20pt}
\begin{tikzpicture}
\draw[black, thick, fill = black!15!white] (-1,0) circle (1);
\node[vertex] at (0,0) {};
\node at (.2,-.3) {$u_2$};
\node at (-1,0) {$G_2$};
\node[vertex] (a1) at (1,0) {};
\draw[black, thick](v) -- (a1);
\end{tikzpicture}
\caption{$G_1^*, G_2^*$}
\label{fig:G*}
\end{figure}

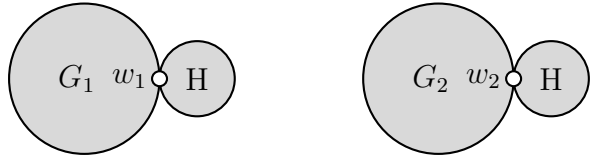
\begin{figure}[!h]
    \centering
\begin{tikzpicture}
\draw[black, thick, fill = black!15!white] (-1,0) circle (1);
\node at (-1.1,0) {$G_1$};
\draw[black, thick, fill = black!15!white](1/2,0) circle (1/2);
\node at (1/2,0) {H};
\node at (-.40,0) {$w_1$};
\node[vertex] at (0,0) {};
\end{tikzpicture}
\hspace{40pt}
\begin{tikzpicture}
\draw[black, thick, fill = black!15!white] (-1,0) circle (1);
\node at (-1.1,0) {$G_2$};
\draw[black, thick, fill = black!15!white](1/2,0) circle (1/2);
\node at (1/2,0) {H};
\node at (-.40,0) {$w_2$};
\node[vertex] at (0,0) {};
\end{tikzpicture}
\caption{Vertex identification of $G_1$ and $G_2$ with $H$, where $w_i$ is the merged vertex of $u_i, v$.}
\label{fig: vertex id}
\end{figure}

\begin{proof}
    Let $G_1 H$ and $G_2 H$ be shorthand for $G_1 \vxi{u_1}{v} H$ and $G_2 \vxi{u_2}{v} H$, respectively. Let $w_1$ be the result of the vertex-identification of $u_1$ and $v$ and $w_2$ be the result of $u_2$ and $v$. Notice that $\Dq^{G_1 H}$ can be written with the following block form
    \[
        \left(
        \begin{array}{c|c|c}
            (\Dq^{H})_{[v]} & \vec x & \vec x \vec y_1^T \\
            \hline
            \vec x^T & 1 & \vec y_1^T \\
            \hline
            \vec y_1 \vec x^T & \vec y_1 & (\Dq^{G_1})_{[u_1]}
        \end{array}
        \right),
    \]
    where $M_{[t]}$ denotes the matrix $M$ with the $t$-th row and column deleted. The middle column and row correspond to vertex $w_1$. The term in the upper right and bottom left blocks come from noting that the distance between vertices $p \in V(G_1)$ and $q \in V(H)$ is $\dist_{G_1}(p,u_1) + \dist_H(v,q)$.  Consider $\det(xI - \Dq^{G_1 H})$,
    \[\det(xI - \Dq^{G_1 H}) = 
        \det \left(
        \begin{array}{c|c|c}
            xI - (\Dq^{H})_{[v]} & -\vec x & - \vec x\vec y_1^T\\
            \hline
            -\vec x^T & x-1 & - \vec y_1^T\\
            \hline
            -\vec y_1 \vec x^T  & -\vec y_1 & xI - (\Dq^{G_1})_{[u_1]}
        \end{array}
        \right)\\.
    \]
    Since adding a multiple of one row or column to another preserves the determinant, we can add the appropriate multiples of the middle row and column to the first rows and columns to get blocks of $0$s in the upper right and lower left blocks. Thus,
    \[\det(xI - \Dq^{G_1 H}) = 
        \det\left(
        \begin{array}{c|c|c}
           S & -x{\cdot}{\vec x} & O\\
            \hline
            -x{\cdot} \vec x^T & x-1 & -\vec y_1^T \\
            \hline
            O & -{\vec y_1} & xI - (\Dq^{G_1})_{[u_1]}
        \end{array}
        \right)\\,
    \]
    where $S$ is the matrix that results from the aforementioned operations. The matrix $S$ is the result of entries that depend solely on $H$ and are thus independent of whether we are working with $G_1$ or $G_2$. 
    Similarly,
    \[
        M= 
        \left(
        \begin{array}{c|c}
            S & -x{\cdot}{\vec x} \\
            \hline
            -x{\cdot}\vec x^T & x-1 \\
        \end{array}
        \right)\\,
    \]
    is also independent of $G_1$. Notice that 
    \[
        xI - \Dq^{G_1} = 
        \left(
        \begin{array}{c|c}
            x-1 & -\vec y_1^T \\
            \hline
            -{\vec y_1} & xI - (\Dq^{G_1})_{[u_1]}
        \end{array}
        \right)\\.
    \]
    We know that the determinant can be expressed in terms of permutations, so consider the possible terms of $\det(xI- \Dq^{G_1 H})$. If the selection for the middle column is $x-1$, the only nonzero terms involve selections from $S$ and $xI - (\Dq^{G_1})_{[u_1]}$. If the selection for the middle column is from $-x{\cdot}{\vec x}$, the selection for the middle row must be from $-x{\cdot}\vec x^T$ if the term is nonzero. Similarly, the  remaining terms involve selections from $-\vec y_1^T$ and $-\vec y_1$ for the middle row and column. Thus, if we compensate for the double-counting of the terms involving $(x-1)$, we can see that
    \[
        \begin{split}
            \det(xI - \Dq^{G_1 H}) &= 
            \det\left(
            \begin{array}{c|c|c}
            S & -x{\cdot}{\vec x} & O\\
            \hline
            -x{\cdot}\vec x^T & x-1 & -\vec y_1^T \\
            \hline
            O & -{\vec y_1} & xI - (\Dq^{G_1})_{[u_1]}
            \end{array}
            \right)\\
            &=
            \det\left(
            \begin{array}{c|c}
            M &  O\\
            \hline
            O & xI - (\Dq^{G_1})_{[u_1]}
            \end{array}
            \right)\\
            &\phantom{{}={}} + 
            \det\left(
            \begin{array}{c|c}
            S & O \\
            \hline
            O & xI - \Dq^{G_1} \\
            \end{array}
            \right)\\
            &\phantom{{}={}} - 
            (x-1)\det(S)\det(xI-(\Dq^{G_1})_{[u_1]})\\
            &=\det(M)\det(xI-(\Dq^{G_1})_{[u_1]}) + \det(S)\det(xI - \Dq^{G_1})\\
            &\phantom{{}={}} - 
            (x-1)\det(S)\det(xI-(\Dq^{G_1})_{[u_1]}).
        \end{split}
    \]
    Now let us consider the characteristic polynomial of $G_2 H$. 
    Since $S$, $M$ depend solely on $\Dq^H$, we can repeat the previous calculations for $G_2 H$ to get 
    \begin{align*}
        \det(xI - \Dq^{G_2 H}) &=\det(M)\det(xI-(\Dq^{G_2})_{[u_2]}) + \det(S)\det(xI - \Dq^{G_2})\\
        &\phantom{{}={}} - 
        (x-1)\det(S)\det(xI-(\Dq^{G_2})_{[u_2]}).
    \end{align*}
    We know $G_1, G_2$ are $\Dq$-cospectral, so $\det(xI - \Dq^{G_1}) = \det(xI - \Dq^{G_2})$. Thus, the only term that could possibly differ between $\det(xI - \Dq^{G_1 H})$ and $\det(xI - \Dq^{G_2 H})$, are the terms involving $\det(xI-(\Dq^{G_1})_{[u_1]})$ and $\det(xI-(\Dq^{G_2})_{[u_2]})$. Showing these two values are equal is enough to complete the proof. 
    
    To see that this is the case, recall that we also assumed $G_1^*, G_2^*$ were $\Dq$-cospectral.  So using the argument as above, the following are equal:
    \begin{align*}
\det(xI - \Dq^{G_1^*}) &= (x-1)\det(xI - \Dq^{G_1}) - (qx)^2\det(xI - (\Dq^{G_1})_{[u_1]}),\text{ and}\\
\det(xI - \Dq^{G_2^*}) &= (x-1)\det(xI - \Dq^{G_2}) - (qx)^2\det(xI - (\Dq^{G_2})_{[u_2]}).
    \end{align*}
    Which is only possible if $\det(xI - (\Dq^{G_1})_{[u_1]})=\det(xI - (\Dq^{G_2})_{[u_2]})$, completing the proof.
\end{proof}

We comment that if $v\in V(G)$, then it is not always the case $(\Dq^G)_{[v]}=\Dq^{G\setminus v}$ since removing a vertex $v$ can impact the distances between other pairs of vertices when the shortest paths must pass through $v$.  Also, we note that in the statement of the previous theorem $G_1^*,G_2^*$ could have been replaced by the vertex identification of $G_1,G_2$ with any connected graph on two or more vertices.

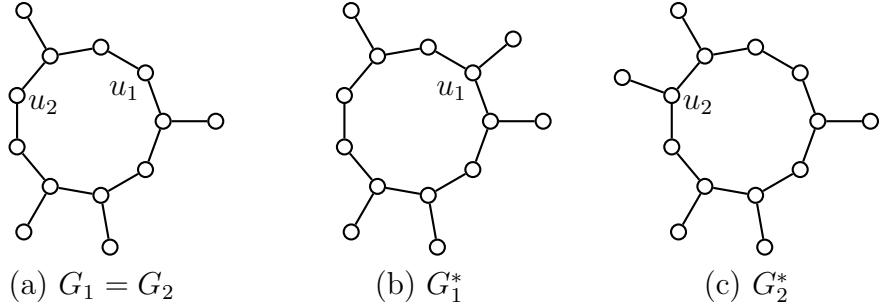
\begin{figure}[h!]
    \centering
\begin{tabular}{c@{\qquad}c@{\qquad}c}
\begin{tikzpicture}[scale=0.7]
\node[vertex] (a1) at (0:1.43) {};
\node[vertex] (a2) at (40:1.43) {};
\node[vertex] (a3) at (80:1.43) {};
\node[vertex] (a4) at (120:1.43) {};
\node[vertex] (a5) at (160:1.43) {};
\node[vertex] (a6) at (200:1.43) {};
\node[vertex] (a7) at (240:1.43) {};
\node[vertex] (a8) at (280:1.43) {};
\node[vertex] (a9) at (320:1.43) {};
\node[vertex] (b1) at (0:2.43) {};
\node[vertex] (b4) at (120:2.43) {};
\node[vertex] (b7) at (240:2.43) {};
\node[vertex] (b8) at (280:2.43) {};
\draw[thick, black] (a1) -- (a2) -- (a3) -- (a4) -- (a5) -- (a6) -- (a7) -- (a8) -- (a9) -- (a1) (b1) -- (a1) (b4) -- (a4) (b7) -- (a7) (b8) -- (a8);
\node at (40:0.9) {$u_1$};
\node at (160:0.9) {$u_2$};
\node[draw, thick,color=white,circle, inner sep = 2pt] at (160:2.43) {};
\end{tikzpicture}
&
\begin{tikzpicture}[scale=0.7]
\node[vertex] (a1) at (0:1.43) {};
\node[vertex] (a2) at (40:1.43) {};
\node[vertex] (a3) at (80:1.43) {};
\node[vertex] (a4) at (120:1.43) {};
\node[vertex] (a5) at (160:1.43) {};
\node[vertex] (a6) at (200:1.43) {};
\node[vertex] (a7) at (240:1.43) {};
\node[vertex] (a8) at (280:1.43) {};
\node[vertex] (a9) at (320:1.43) {};
\node[vertex] (b1) at (0:2.43) {};
\node[vertex] (b2) at (40:2.43) {};
\node[vertex] (b4) at (120:2.43) {};
\node[vertex] (b7) at (240:2.43) {};
\node[vertex] (b8) at (280:2.43) {};
\draw[thick, black] (a1) -- (a2) -- (a3) -- (a4) -- (a5) -- (a6) -- (a7) -- (a8) -- (a9) -- (a1) (b1) -- (a1) (b2) -- (a2) (b4) -- (a4) (b7) -- (a7) (b8) -- (a8);
\node at (40:0.9) {$u_1$};
\node[draw, thick,color=white,circle, inner sep = 2pt] at (160:2.43) {};

\end{tikzpicture}
&
\begin{tikzpicture}[scale=0.7]
\node[vertex] (a1) at (0:1.43) {};
\node[vertex] (a2) at (40:1.43) {};
\node[vertex] (a3) at (80:1.43) {};
\node[vertex] (a4) at (120:1.43) {};
\node[vertex] (a5) at (160:1.43) {};
\node[vertex] (a6) at (200:1.43) {};
\node[vertex] (a7) at (240:1.43) {};
\node[vertex] (a8) at (280:1.43) {};
\node[vertex] (a9) at (320:1.43) {};
\node[vertex] (b1) at (0:2.43) {};
\node[vertex] (b4) at (120:2.43) {};
\node[vertex] (b5) at (160:2.43) {};
\node[vertex] (b7) at (240:2.43) {};
\node[vertex] (b8) at (280:2.43) {};
\draw[thick, black] (a1) -- (a2) -- (a3) -- (a4) -- (a5) -- (a6) -- (a7) -- (a8) -- (a9) -- (a1) (b1) -- (a1) (b4) -- (a4) (b5) -- (a5) (b7) -- (a7) (b8) -- (a8);
\node at (160:0.9) {$u_2$};
\end{tikzpicture}\\
(a) $G_1=G_2$&
(b) $G_1^*$&
(c) $G_2^*$
\end{tabular}
\caption{Illustration of Theorem~\ref{glue}}
\label{fig:cyclic graphs}
\end{figure}

The technique from Theorem~\ref{glue} can be used to construct many infinite cospectral families. As a simple example, consider the graph in Figure~\ref{fig:cyclic graphs}. If we let both $G_1,G_2$ be the graph $G$ from part (a) in Figure~\ref{fig:cyclic graphs}, and $G_1^*, G_2^*$ be the graphs from parts (b) and (c) in Figure~\ref{fig:cyclic graphs}, then we can see that $G_1, G_2$ are isomorphic, as are $G_1^*, G_2^*$. Thus, by Theorem~\ref{glue} we can attach any graph $H$ to vertices $u_1, u_2$ to create a pair of non-isomorphic cospectral graphs. For example let $H$ be an arbitrary path, then the following graphs in Figure~\ref{fig:cyclepath} are a family of non-isomorphic cospectral graphs.

\begin{figure}[!h]
    \centering
\begin{tikzpicture}[scale=0.7]
\node[vertex] (a1) at (0:1.43) {};
\node[vertex] (a2) at (40:1.43) {};
\node[vertex] (a3) at (80:1.43) {};
\node[vertex] (a4) at (120:1.43) {};
\node[vertex] (a5) at (160:1.43) {};
\node[vertex] (a6) at (200:1.43) {};
\node[vertex] (a7) at (240:1.43) {};
\node[vertex] (a8) at (280:1.43) {};
\node[vertex] (a9) at (320:1.43) {};
\node[vertex] (b1) at (0:2.43) {};
\node[vertex] (b2) at (40:2.43) {};
\node[vertex] (b4) at (120:2.43) {};
\node[vertex] (b7) at (240:2.43) {};
\node[vertex] (b8) at (280:2.43) {};
\draw[thick, black] (a1) -- (a2) -- (a3) -- (a4) -- (a5) -- (a6) -- (a7) -- (a8) -- (a9) -- (a1) (b1) -- (a1) (b2) -- (a2) (b4) -- (a4) (b7) -- (a7) (b8) -- (a8);
\node at (40:0.9) {$u_1$};
\node[vertex](c2) at (40:4.43) {};
\draw[thick, dashed](b2) -- (c2);
\end{tikzpicture}
\hfil
\begin{tikzpicture}[scale=0.7]
\node[vertex] (a1) at (0:1.43) {};
\node[vertex] (a2) at (40:1.43) {};
\node[vertex] (a3) at (80:1.43) {};
\node[vertex] (a4) at (120:1.43) {};
\node[vertex] (a5) at (160:1.43) {};
\node[vertex] (a6) at (200:1.43) {};
\node[vertex] (a7) at (240:1.43) {};
\node[vertex] (a8) at (280:1.43) {};
\node[vertex] (a9) at (320:1.43) {};
\node[vertex] (b1) at (0:2.43) {};
\node[vertex] (b4) at (120:2.43) {};
\node[vertex] (b5) at (160:2.43) {};
\node[vertex] (b7) at (240:2.43) {};
\node[vertex] (b8) at (280:2.43) {};
\draw[thick, black] (a1) -- (a2) -- (a3) -- (a4) -- (a5) -- (a6) -- (a7) -- (a8) -- (a9) -- (a1) (b1) -- (a1) (b4) -- (a4) (b5) -- (a5) (b7) -- (a7) (b8) -- (a8);
\node at (160:0.9) {$u_2$};
\node[vertex](c5) at (160:4.43) {};
\draw[thick, dashed](c5) -- (b5);
\end{tikzpicture}
\caption{Vertex identification of $G$ with an arbitrary path}
\label{fig:cyclepath}

\end{figure}
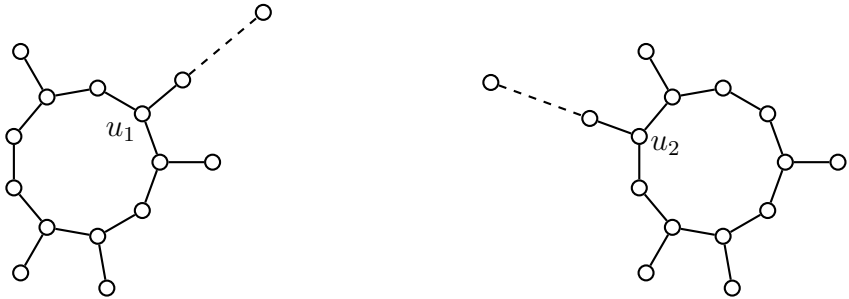
Theorem~\ref{glue} also allows us to prove a result about cospectrality for trees. McKay \cite{mckay} showed that almost all trees have a cospectral mate for the standard distance matrix. We are able to prove the same result for the exponential distance matrix. 
\begin{thm}
    Almost all trees have a $\Dq$-cospectral mate.
\end{thm}
\begin{proof}
   Let $T_1, T_2$ both be the tree from Figure~\ref{fig:cospec trees 1}(a). Let $T_1^*, T_2^*$ be the trees from Figure~\ref{fig:cospec trees 1}(b) and~\ref{fig:cospec trees 1}(c), respectively. Again, $T_1,T_2$ are the same graph so they clearly have the same $\Dq$ characteristic polynomial. $T_1^*,T_2^*$ are also $\Dq$-cospectral, which is computationally confirmed. Thus, we can apply Theorem~\ref{glue}, with $H$ being any tree. The proportion of trees on $n$ vertices which contain $T$ glued at $u_1$ tends toward $1$ as $n$ becomes large, these can be swapped out for a copy of $T$ glued at $u_2$ (see \cite{schwenk}). In particular, most trees will have a non-isomorphic cospectral mate. 
\end{proof}

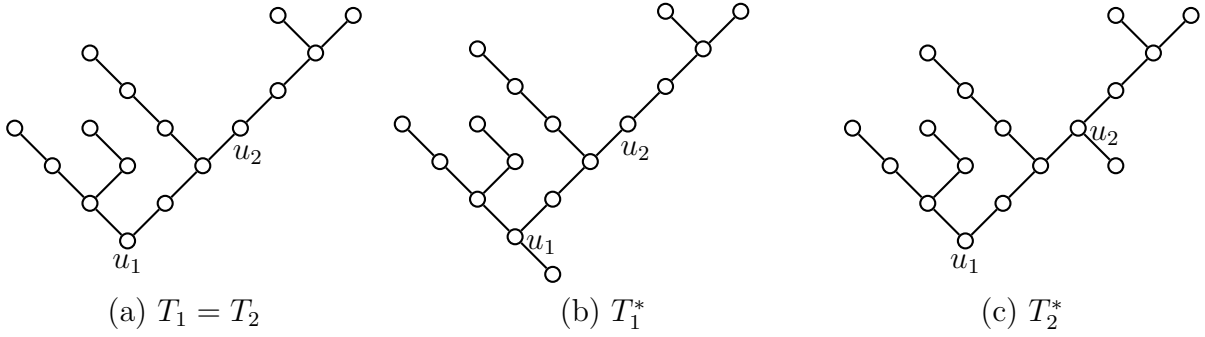
\begin{figure}[!h]
    \centering
    \begin{tabular}{ccc}
    \begin{tikzpicture}[scale = 0.5, rotate=0]
\node[vertex] (4) at (0,0) {};
\node at (0,-.6) {$u_1$};
\node[vertex] (3) at (-1,1) {};
\draw[black, thick] (3) -- (4);
\node[vertex] (5) at (1,1) {};
\draw[black, thick] (4) -- (5);
\node[vertex](2) at (-2,2) {};
\node[vertex] (1) at (-3,3) {};
\draw[black, thick] (3) -- (2);
\draw[black, thick] (2) -- (1);
\node[vertex] (11) at (0,2) {};
\node[vertex] (12) at (-1,3) {};
\draw[thick, black] (3) -- (11);
\draw[black, thick] (11) -- (12);
\node[vertex] (6) at (2,2) {};
\draw[black, thick] (5) -- (6);
\node[vertex] (7) at (3,3) {};
\node at (3.2,2.3) {$u_2$};
\draw[black, thick] (6) -- (7);
\node[vertex] (8) at (4,4) {};
\draw[thick, black] (7) -- (8);
\node[vertex] (9) at (5,5) {};
\draw[thick, black] (8) -- (9);
\node[vertex] (10) at (6,6) {};
\draw[black, thick] (9) -- (10);
\node[vertex] (16) at (4,6) {};
\draw[thick, black](9) -- (16);
\node[vertex] (13) at (1,3) {};
\draw[black, thick] (6) -- (13);
\node[vertex] (14) at (0,4) {};
\draw[black, thick] (13) -- (14);
\node[vertex] (15) at (-1, 5) {};
\draw[thick, black] (14) -- (15);
\end{tikzpicture}
&
\begin{tikzpicture}[scale = 0.5,rotate=0]
\node[vertex] (4) at (0,0) {};
\node at (0.7,-.2) {$u_1$};
\node[vertex] (3) at (-1,1) {};
\draw[black, thick] (3) -- (4);
\node[vertex] (5) at (1,1) {};
\draw[black, thick] (4) -- (5);
\node[vertex](2) at (-2,2) {};
\node[vertex] (1) at (-3,3) {};
\draw[black, thick] (3) -- (2);
\draw[black, thick] (2) -- (1);
\node[vertex] (11) at (0,2) {};
\node[vertex] (12) at (-1,3) {};
\draw[thick, black] (3) -- (11);
\draw[black, thick] (11) -- (12);
\node[vertex] (6) at (2,2) {};
\draw[black, thick] (5) -- (6);
\node[vertex] (7) at (3,3) {};
\node at (3.2,2.3) {$u_2$};
\draw[black, thick] (6) -- (7);
\node[vertex] (8) at (4,4) {};
\draw[thick, black] (7) -- (8);
\node[vertex] (9) at (5,5) {};
\draw[thick, black] (8) -- (9);
\node[vertex] (10) at (6,6) {};
\draw[black, thick] (9) -- (10);
\node[vertex] (16) at (4,6) {};
\draw[thick, black](9) -- (16);
\node[vertex] (13) at (1,3) {};
\draw[black, thick] (6) -- (13);
\node[vertex] (14) at (0,4) {};
\draw[black, thick] (13) -- (14);
\node[vertex] (15) at (-1, 5) {};
\draw[thick, black] (14) -- (15);
\node[vertex] (16) at (1,-1) {};
\draw[thick, black] (4) -- (16);
\end{tikzpicture}
\hspace{20pt}
&
\begin{tikzpicture}[scale = 0.5, rotate = 0]
\node[vertex] (4) at (0,0) {};
\node at (0,-.6) {$u_1$};
\node[vertex] (3) at (-1,1) {};
\draw[black, thick] (3) -- (4);
\node[vertex] (5) at (1,1) {};
\draw[black, thick] (4) -- (5);
\node[vertex](2) at (-2,2) {};
\node[vertex] (1) at (-3,3) {};
\draw[black, thick] (3) -- (2);
\draw[black, thick] (2) -- (1);
\node[vertex] (11) at (0,2) {};
\node[vertex] (12) at (-1,3) {};
\draw[thick, black] (3) -- (11);
\draw[black, thick] (11) -- (12);
\node[vertex] (6) at (2,2) {};
\draw[black, thick] (5) -- (6);
\node[vertex] (7) at (3,3) {};
\node at (3.7,2.8) {$u_2$};
\draw[black, thick] (6) -- (7);
\node[vertex] (8) at (4,4) {};
\draw[thick, black] (7) -- (8);
\node[vertex] (9) at (5,5) {};
\draw[thick, black] (8) -- (9);
\node[vertex] (10) at (6,6) {};
\draw[black, thick] (9) -- (10);
\node[vertex] (16) at (4,6) {};
\draw[thick, black](9) -- (16);
\node[vertex] (13) at (1,3) {};
\draw[black, thick] (6) -- (13);
\node[vertex] (14) at (0,4) {};
\draw[black, thick] (13) -- (14);
\node[vertex] (15) at (-1, 5) {};
\draw[thick, black] (14) -- (15);
\node[vertex] (17) at (4,2) {};
\draw[black, thick] (7) -- (17);
\end{tikzpicture}\\
(a) $T_1 = T_2$ & (b) $T_1^*$ & (c) $T_2^*$
     \end{tabular}
      \caption{Constructing cospectral trees, taken from \cite{mckay}}
      \label{fig:cospec trees 1}
\end{figure}

\subsection{Exponential Distance Switching}


Following the cospectral construction for the distance matrix from Heysse \cite{heysse}, we can create a similar construction for the exponential distance matrix. However, unlike the result from Heysse \cite{heysse}, we must restrict to graphs with diameter two. Suppose a graph $G$ has one of the graphs in Figure~\ref{fig:switching subgraphs} as an induced subgraph. 
\begin{figure}[!h]
    \centering
    \begin{tikzpicture}
\node[vertex] (g1) at (0,0) {};
\node at (-.33,0) {$g_1$};
\node[vertex] (h1) at (2.25,0) {};
\node at (2.58,0) {$h_1$};
\node[vertex] (g2) at (0,-2.25) {};
\node at (-.33,-2.25) {$g_2$};
\node[vertex] (h2) at (2.25,-2.25) {};
\node at (2.58,-2.25) {$h_2$};
\node[vertex] (s) at (1,-3) {};
\node at (1,-3.3) {$S$};
\draw[black, thick](g1) -- (s)--(g2)--(g1)--(h1)--(h2)--(g2)--(h1);
\draw[black, thick](g1) -- (h2);
\end{tikzpicture}
\hfill
\begin{tikzpicture}
\node[vertex] (g1) at (0,0) {};
\node at (-.33,0) {$g_1$};
\node[vertex] (h1) at (2.25,0) {};
\node at (2.58,0) {$h_1$};
\node[vertex] (g2) at (0,-2.25) {};
\node at (-.33,-2.25) {$g_2$};
\node[vertex] (h2) at (2.25,-2.25) {};
\node at (2.58,-2.25) {$h_2$};
\node[vertex] (s) at (1,-3) {};
\node at (1,-3.3) {$S$};
\draw[black, thick](g1) -- (s)--(g2)--(g1)--(h1)--(h2)--(g2);
\end{tikzpicture}
\hfill
\begin{tikzpicture}
\node[vertex] (g1) at (0,0) {};
\node at (-.33,0) {$g_1$};
\node[vertex] (h1) at (2.25,0) {};
\node at (2.58,0) {$h_1$};
\node[vertex] (g2) at (2.25,-2.25) {};
\node at (2.58,-2.25) {$g_2$};
\node[vertex] (h2) at (0,-2.25) {};
\node at (-.33,-2.25) {$h_2$};
\node[vertex] (s) at (1,-3) {};
\node at (1,-3.3) {$S$};
\draw[black, thick](g1) -- (s)--(g2)--(h1)--(g1)--(h2)--(g2);
\end{tikzpicture}
\caption{Subgraph switching candidates from Heysse \cite{heysse}.}
\label{fig:switching subgraphs}
\end{figure}
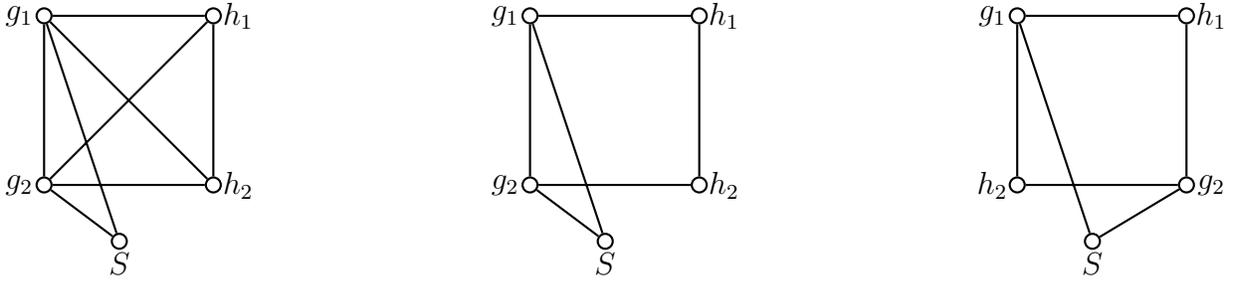

Let $S$ be the set of vertices $\{s\mid s \sim g_1, g_2 \text{ and } s\not \sim h_1, h_2\}$.  The set $S$ is not restricted to being a single vertex; a simple example involves creating twin copies of the vertex labeled $S$ in Figure~\ref{fig:switching subgraphs}. The vertices within $S$ can be arbitrarily connected with each other.  

Further, suppose we can partition the vertices in $V(G) \setminus \{g_1,g_2, h_1,h_2\}$ into two sets, $A$ and $B$, with $v \in A$ if and only if
\begin{align*}
    \dist_G(v,g_1) + \dist_G(v,g_2) - \dist_G(v,h_1) - \dist_G(v,h_2) = -2. 
\end{align*}
Since we have restricted the diameter of $G$ to be two, it follows that 
\begin{align*}
     \dist_G(v,g_1) = \dist_G(v,g_2) = 1 \text{ and } \dist_G(v,h_1) = \dist_G(v,h_2) = 2.
\end{align*}
And we have $v \in B$ if and only if
\begin{align*}
     \dist_G(v,g_1) + \dist_G(v,g_2) - \dist_G(v,h_1) - \dist_G(v,h_2) = 0. 
\end{align*}
Notice that $S \subseteq A$. From our graph $G$, define a graph $H$ such that 
\begin{align*}
    V(H) = V(G), \,E(H) = E(G) \setminus \{\{s,g_i\}\mid s\in S, i \in \{1,2\}\} \cup \{\{s,h_i\}\mid s\in S, i \in \{1,2\}\}.
\end{align*}
\begin{thm}
    Suppose $G$ and $H$ are graphs satisfying the conditions above.  Moreover, if  for all $v \in B$, $\dist_H(v,u) = \dist_G(v,u)$ for all $u \in V$ and if for all $w \in A$, $\dist_H(w,u) = \dist_G(w,u)$ for all $u \in V(G) \setminus \{g_1,g_2, h_1,h_2\}$ and 
    \begin{align*}
        \dist_H(w,g_i) = \dist_G(w,g_i) + 1 \textbf{ and } \dist_H(w,h_i) = \dist_G(w,h_i) -1
    \end{align*}
    for $i \in {1,2}$, then $G$ and $H$ are exponential distance cospectral. 
\end{thm}

\begin{proof}[Sketch of proof based on Heysse \cite{heysse}]
    We  show how to perturb an eigenvector for the graph $G$ to produce an eigenvector for the graph $H$ with the same eigenvalue.
    
    So suppose $(\lambda,  \vec{x})$ is an eigenpair of $\Dq^G$ and $\lambda \ne -k$ (where $k = q-1$,  $k=q^2-1$, or $k=-(q^2-2q+1)$ for the first, second, and third graphs shown in  Figure~\ref{fig:switching subgraphs}, respectively). 
    We construct $\vec{y} \coloneqq \vec{x} + \Delta$ where
    \[
    \Delta_i= 
        \begin{cases} 
        0 & \text{if } i \not \in \{ g_1, g_2, h_1, h_2\},\\
         \frac{(q^2-q)\sum_{j\in A} x_j}{k+\lambda} & \text{if }i\in \{g_1, g_2\},\\
        -\frac{(q^2-q)\sum_{j\in A} x_j}{k+\lambda} & \text{if }i\in \{h_1, h_2\}.
        \end{cases}
    \]
    
    We have $(\lambda,\vec{y})$ is an eigenpair for  $\Dq^H$.  This is verified in a manner similar to \cite{heysse}, where the term $c(v)$ which is used in \cite{heysse} when verifying that the new vector is an eigenvector now becomes
    \[
    c(v) = q^{\dist_G(v,g_1)} + q^{\dist_G(v,g_2)} - q^{\dist_H(v,h_1)} - q^{\dist_H(v,h_2)}.
    \]
    This perturbation preserves dimension of the eigenspaces for $\lambda\ne -k$ since this process can be reversed.  Finally, we note that since we know that the eigenvalues (including multiplicity) which are not $-k$ are the same for both graphs, then we can conclude that the multiplicity of $-k$ as an eigenvalue is the same for both graphs.  So $G$ and $H$ are $\Dq$-cospectral.
\end{proof}

On seven vertices, there are eleven pairs of cospectral graphs for $\mathcal{D}$ and eleven pairs for $\Dq$. Ten pairs have diameter two and satisfying the switching requirements and so are cospectral for both.  The remaining pairs are shown in Figure~\ref{fig:Dq not D} and Figure~\ref{fig:D but not Dq}. Notice that both of these pairs have diameter three.

\begin{figure}[htb]
    \centering
\begin{tikzpicture}
\node[vertex](1) at (0,0) {};
\node[vertex](2) at (1,0) {};
\node[vertex](3) at (0,-1) {};
\node[vertex](4) at (1,-1) {};
\node[vertex](5) at (2,-1) {};
\node[vertex](6) at (1,-2) {};
\node[vertex](7) at (2,-2) {};
\draw[thick](1) -- (2) (1) -- (3) (2) -- (4) (1) -- (4) (3) -- (4) -- (5) (4) -- (6) (5) -- (7);
\end{tikzpicture}
\hfil
\begin{tikzpicture}
\node[vertex](1) at (0,0) {};
\node[vertex](2) at (1,0) {};
\node[vertex](3) at (-1,-1) {};
\node[vertex](4) at (0,-1) {};
\node[vertex](5) at (1,-1) {};
\node[vertex](6) at (0,-2) {};
\node[vertex](7) at (1,-2) {};
\draw[thick](1) -- (2) (1) -- (4) (2) -- (5) (1) -- (5) (3) -- (4) -- (5) (4) -- (6) (5) -- (7);
\end{tikzpicture}
    \caption{$\Dq$-cospectral, but not $D$-cospectral.}
    \label{fig:Dq not D}

\begin{tikzpicture}
\node[vertex](1) at (0,0) {};
\node[vertex](2) at (1,0) {};
\node[vertex](3) at (2,0) {};
\node[vertex](4) at (0,-1) {};
\node[vertex](5) at (1,-1) {};
\node[vertex](6) at (2,-1) {};
\node[vertex](7) at (0.5,-2) {};
\draw[thick](1) -- (4) (1) -- (5) (1) -- (6) (2) -- (4) (2) -- (5) (2) -- (6) (3) -- (5) (3) -- (6) (4) -- (5) (4) -- (7) (5) -- (7);
\end{tikzpicture}
\hfil
\begin{tikzpicture}
\node[vertex](1) at (0,0) {};
\node[vertex](2) at (1,0) {};
\node[vertex](3) at (2,0) {};
\node[vertex](4) at (0,-1) {};
\node[vertex](5) at (1,-1) {};
\node[vertex](6) at (2,-1) {};
\node[vertex](7) at (0.5,-2) {};
\draw[thick](1) -- (2) -- (3) (1) -- (4) (1) -- (6) (2) -- (5) (3) -- (4) (3) -- (5) (3) -- (6) (4) -- (5) (4) -- (7) (5) -- (7);
\end{tikzpicture}
    \caption{$D$-cospectral, but not $\Dq$-cospectral.}
    \label{fig:D but not Dq}
\end{figure}
\subsection{A Cospectral Unicyclic Family}

\begin{figure}[htb]
\begin{center}
\begin{tabular}{c}
\begin{tikzpicture}
\node[vertex] (a1) at (0,0) {};
\node[vertex] (a2) at (1,0) {};
\node[vertex] (a3) at (0,1) {};
\node[vertex] (a4) at (1,1) {};
\node[vertex] (a5) at (-1,0) {};

\node[vertex] (b1) at (2, 1) {};
\node[above] at (3,1.25) {\small $k{+}1$ vertices};
\draw[thick,dotted,rounded corners] (1.75,0.75) rectangle (4.25,1.25);
\node[vertex] (b3) at (4,1) {};

\node[vertex] (c1) at (2, 0) {};
\node[below] at (3,-0.25) {\small $k{+}3$ vertices};
\draw[thick,dotted,rounded corners] (1.75,-0.25) rectangle (4.25,0.25);
\node[vertex] (c3) at (4,0) {};

\node[vertex] (d1) at (1, -1) {};
\node[rotate = 90] at (0.4,-2) {\small $k{+}2$ vertices};
\draw[thick,dotted,rounded corners] (0.75,-0.75) rectangle (1.25,-3.25);
\node[vertex] (d3) at (1,-3) {};

\draw[thick] 
(a5)--(a1)--(a2)--(a4)--(a3)--(a1)
(d1)--(a2)--(c1)
(a4)--(b1)
;

\draw[thick,dashed]
(b1)--(b3)
(c1)--(c3)
(d1)--(d3);
\end{tikzpicture}\\
$(a)$ The family of graphs $G_1$.
\end{tabular}
\begin{tabular}{c}
\begin{tikzpicture}
\node[vertex] (a1) at (0,0) {};
\node[vertex] (a2) at (1,0) {};
\node[vertex] (a3) at (0,1) {};
\node[vertex] (a4) at (1,1) {};
\node[vertex] (a5) at (-1,0) {};
\node[vertex] (a6) at (-2,0) {};

\node[vertex] (b1) at (2, 1) {};
\node[above] at (3,1.25) {\small $k{+}1$ vertices};
\draw[thick,dotted,rounded corners] (1.75,0.75) rectangle (4.25,1.25);
\node[vertex] (b3) at (4,1) {};

\node[vertex] (c1) at (2, 0) {};
\node[below] at (3,-.25) {\small $k{+}2$ vertices};
\draw[thick,dotted,rounded corners] (1.75,-0.25) rectangle (4.25,0.25);
\node[vertex] (c3) at (4,0) {};

\node[vertex] (d1) at (0, -1) {};
\node[rotate = 90] at (-0.6,-2) {\small $k{+}2$ vertices};
\draw[thick,dotted,rounded corners] (0.25,-0.75) rectangle (-0.25,-3.25);
\node[vertex] (d3) at (0,-3) {};

\draw[thick] 
(a5)--(a1)--(a2)--(a4)--(a3)--(a1)--(d1)
(a2)--(c1)
(a4)--(b1)
(a6)--(a5);

\draw[thick,dashed]
(b1)--(b3)
(c1)--(c3)
(d1)--(d3);
\end{tikzpicture}\\
$(b)$ The family of graphs $G_2$.
\end{tabular}
\caption{A cospectral family of graphs}
\label{Fig:OrangeFamily}
\end{center}
\end{figure}

Thus far we have shown that almost all trees have a cospectral mate and given a construction for cospectral graphs with diameter at most $2$. We will now give another family of cospectral graphs that are unicyclic.

To prove this family is cospectral we will be using graph reductions presented in Corollary~\ref{Cor:fkRecursion} to reduce our family for any value of $k$ into ten cases to check. These cases are handled computationally.  

\begin{prop}\label{prop:OrangeFamily}
The graphs shown in Figure~\ref{Fig:OrangeFamily} are $\Dq$-cospectral for all $k \geq 0$. 
\end{prop}

\begin{proof}
We repeatedly use Corollary~\ref{Cor:fkRecursion} to reduce the graphs down to a linear combination of a small number of fixed graphs (with coefficients polynomials in the $f_i,x,q$).  We illustrate this process for Figure~\ref{Fig:OrangeFamily}(a).  In the following we will use a graph to represent its characteristic polynomial and where a solid line, 
\begin{tikzpicture}[scale=0.6]
\node[tiny_vertex] (a) at (0,0) {};
\node[tiny_vertex] (b) at (1,0) {};
\draw[thick] (a)--(b);
\end{tikzpicture}, 
indicates an edge in the graph and a dotted line, 
\begin{tikzpicture}[scale=0.6]
\node[tiny_vertex] (a) at (0,0) {};
\node[tiny_vertex] (b) at (1,0) {};
\draw[ultra thick,dotted] (a)--(b);
\end{tikzpicture}, 
indicates a path which has not yet had Corollary~\ref{Cor:fkRecursion} applied.
\begin{align*}
\left(\hspace{-8pt}\begin{array}{c}
\begin{tikzpicture}[scale=0.6]
\node[tiny_vertex] (a) at (-1,0) {};
\node[tiny_vertex] (b) at (0,0) {};
\node[tiny_vertex] (c) at (1,0) {};
\node[tiny_vertex] (d) at (2,0) {};
\node[tiny_vertex] (e) at (0,1) {};
\node[tiny_vertex] (f) at (1,1) {};
\node[tiny_vertex] (g) at (2,1) {};
\node[tiny_vertex] (h) at (1,-1) {};
\draw[thick] (a)--(b)--(c)--(f)--(e)--(b);
\draw[ultra thick,dotted] (c)--(d);
\draw[ultra thick,dotted] (f)--(g);
\draw[ultra thick,dotted] (c)--(h);
\end{tikzpicture}
\end{array}\hspace{-3pt}\right)&=f_{k+1}
\left(\hspace{-8pt}\begin{array}{c}
\begin{tikzpicture}[scale=0.6]
\node[tiny_vertex] (a) at (-1,0) {};
\node[tiny_vertex] (b) at (0,0) {};
\node[tiny_vertex] (c) at (1,0) {};
\node[tiny_vertex] (d) at (2,0) {};
\node[tiny_vertex] (e) at (0,1) {};
\node[tiny_vertex] (f) at (1,1) {};
\node[tiny_vertex] (g) at (2,1) {};
\node[tiny_vertex] (h) at (1,-1) {};
\draw[thick] (a)--(b)--(c)--(f)--(e)--(b);
\draw[ultra thick,dotted] (c)--(d);
\draw[thick] (f)--(g);
\draw[ultra thick,dotted] (c)--(h);
\end{tikzpicture}
\end{array}\hspace{-3pt}\right)-q^2x^2f_k
\left(\hspace{-8pt}\begin{array}{c}
\begin{tikzpicture}[scale=0.6]
\node[tiny_vertex] (a) at (-1,0) {};
\node[tiny_vertex] (b) at (0,0) {};
\node[tiny_vertex] (c) at (1,0) {};
\node[tiny_vertex] (d) at (2,0) {};
\node[tiny_vertex] (e) at (0,1) {};
\node[tiny_vertex] (f) at (1,1) {};
\node[tiny_vertex] (h) at (1,-1) {};
\draw[thick] (a)--(b)--(c)--(f)--(e)--(b);
\draw[ultra thick,dotted] (c)--(d);
\draw[ultra thick,dotted] (c)--(h);
\end{tikzpicture}
\end{array}\hspace{-3pt}\right)\\
&=
f_{k+1}f_{k+3}
\left(\hspace{-8pt}\begin{array}{c}
\begin{tikzpicture}[scale=0.6]
\node[tiny_vertex] (a) at (-1,0) {};
\node[tiny_vertex] (b) at (0,0) {};
\node[tiny_vertex] (c) at (1,0) {};
\node[tiny_vertex][tiny_vertex] (d) at (2,0) {};
\node[tiny_vertex] (e) at (0,1) {};
\node[tiny_vertex] (f) at (1,1) {};
\node[tiny_vertex] (g) at (2,1) {};
\node[tiny_vertex] (h) at (1,-1) {};
\draw[thick] (a)--(b)--(c)--(f)--(e)--(b);
\draw[thick] (c)--(d);
\draw[thick] (f)--(g);
\draw[ultra thick,dotted] (c)--(h);
\end{tikzpicture}
\end{array}\hspace{-3pt}\right)
-q^2x^2f_{k+1}f_{k+2}
\left(\hspace{-8pt}\begin{array}{c}
\begin{tikzpicture}[scale=0.6]
\node[tiny_vertex] (a) at (-1,0) {};
\node[tiny_vertex] (b) at (0,0) {};
\node[tiny_vertex] (c) at (1,0) {};
\node[tiny_vertex] (e) at (0,1) {};
\node[tiny_vertex] (f) at (1,1) {};
\node[tiny_vertex] (g) at (2,1) {};
\node[tiny_vertex] (h) at (1,-1) {};
\draw[thick] (a)--(b)--(c)--(f)--(e)--(b);
\draw[thick] (f)--(g);
\draw[ultra thick,dotted] (c)--(h);
\end{tikzpicture}
\end{array}\hspace{-3pt}\right)\\&\phantom{=}
-q^2x^2f_kf_{k+3}
\left(\hspace{-8pt}\begin{array}{c}
\begin{tikzpicture}[scale=0.6]
\node[tiny_vertex] (a) at (-1,0) {};
\node[tiny_vertex] (b) at (0,0) {};
\node[tiny_vertex] (c) at (1,0) {};
\node[tiny_vertex] (d) at (2,0) {};
\node[tiny_vertex] (e) at (0,1) {};
\node[tiny_vertex] (f) at (1,1) {};
\node[tiny_vertex] (h) at (1,-1) {};
\draw[thick] (a)--(b)--(c)--(f)--(e)--(b);
\draw[thick] (c)--(d);
\draw[ultra thick,dotted] (c)--(h);
\end{tikzpicture}
\end{array}\hspace{-3pt}\right)
+q^4x^4f_kf_{k+2}
\left(\hspace{-8pt}\begin{array}{c}
\begin{tikzpicture}[scale=0.6]
\node[tiny_vertex] (a) at (-1,0) {};
\node[tiny_vertex] (b) at (0,0) {};
\node[tiny_vertex] (c) at (1,0) {};
\node[tiny_vertex] (e) at (0,1) {};
\node[tiny_vertex] (f) at (1,1) {};
\node[tiny_vertex] (h) at (1,-1) {};
\draw[thick] (a)--(b)--(c)--(f)--(e)--(b);
\draw[ultra thick,dotted] (c)--(h);
\end{tikzpicture}
\end{array}\hspace{-3pt}\right)\\ 
&=f_{k+1}f_{k+3}f_{k+2}
\left(\hspace{-8pt}\begin{array}{c}
\begin{tikzpicture}[scale=0.6]
\node[tiny_vertex] (a) at (-1,0) {};
\node[tiny_vertex] (b) at (0,0) {};
\node[tiny_vertex] (c) at (1,0) {};
\node[tiny_vertex][tiny_vertex] (d) at (2,0) {};
\node[tiny_vertex] (e) at (0,1) {};
\node[tiny_vertex] (f) at (1,1) {};
\node[tiny_vertex] (g) at (2,1) {};
\node[tiny_vertex] (h) at (1,-1) {};
\draw[thick] (a)--(b)--(c)--(f)--(e)--(b);
\draw[thick] (c)--(d);
\draw[thick] (f)--(g);
\draw[thick] (c)--(h);
\end{tikzpicture}
\end{array}\hspace{-3pt}\right)
-q^2x^2f_{k+1}f_{k+3}f_{k+1}
\left(\hspace{-8pt}\begin{array}{c}
\begin{tikzpicture}[scale=0.6]
\node[tiny_vertex] (a) at (-1,0) {};
\node[tiny_vertex] (b) at (0,0) {};
\node[tiny_vertex] (c) at (1,0) {};
\node[tiny_vertex][tiny_vertex] (d) at (2,0) {};
\node[tiny_vertex] (e) at (0,1) {};
\node[tiny_vertex] (f) at (1,1) {};
\node[tiny_vertex] (g) at (2,1) {};
\draw[thick] (a)--(b)--(c)--(f)--(e)--(b);
\draw[thick] (c)--(d);
\draw[thick] (f)--(g);
\end{tikzpicture}
\end{array}\hspace{-3pt}\right)\\ &\phantom{=}
-q^2x^2f_{k+1}f_{k+2}f_{k+2}
\left(\hspace{-8pt}\begin{array}{c}
\begin{tikzpicture}[scale=0.6]
\node[tiny_vertex] (a) at (-1,0) {};
\node[tiny_vertex] (b) at (0,0) {};
\node[tiny_vertex] (c) at (1,0) {};
\node[tiny_vertex] (e) at (0,1) {};
\node[tiny_vertex] (f) at (1,1) {};
\node[tiny_vertex] (g) at (2,1) {};
\node[tiny_vertex] (h) at (1,-1) {};
\draw[thick] (a)--(b)--(c)--(f)--(e)--(b);
\draw[thick] (f)--(g);
\draw[thick] (c)--(h);
\end{tikzpicture}
\end{array}\hspace{-3pt}\right)
+q^4x^4f_{k+1}f_{k+2}f_{k+1}
\left(\hspace{-8pt}\begin{array}{c}
\begin{tikzpicture}[scale=0.6]
\node[tiny_vertex] (a) at (-1,0) {};
\node[tiny_vertex] (b) at (0,0) {};
\node[tiny_vertex] (c) at (1,0) {};
\node[tiny_vertex] (e) at (0,1) {};
\node[tiny_vertex] (f) at (1,1) {};
\node[tiny_vertex] (g) at (2,1) {};
\draw[thick] (a)--(b)--(c)--(f)--(e)--(b);
\draw[thick] (f)--(g);
\end{tikzpicture}
\end{array}\hspace{-3pt}\right)\\&\phantom{=}
-q^2x^2f_kf_{k+3}f_{k+2}
\left(\hspace{-8pt}\begin{array}{c}
\begin{tikzpicture}[scale=0.6]
\node[tiny_vertex] (a) at (-1,0) {};
\node[tiny_vertex] (b) at (0,0) {};
\node[tiny_vertex] (c) at (1,0) {};
\node[tiny_vertex] (d) at (2,0) {};
\node[tiny_vertex] (e) at (0,1) {};
\node[tiny_vertex] (f) at (1,1) {};
\node[tiny_vertex] (h) at (1,-1) {};
\draw[thick] (a)--(b)--(c)--(f)--(e)--(b);
\draw[thick] (c)--(d);
\draw[thick] (c)--(h);
\end{tikzpicture}
\end{array}\hspace{-3pt}\right)
+q^4x^4f_kf_{k+3}f_{k+1}
\left(\hspace{-8pt}\begin{array}{c}
\begin{tikzpicture}[scale=0.6]
\node[tiny_vertex] (a) at (-1,0) {};
\node[tiny_vertex] (b) at (0,0) {};
\node[tiny_vertex] (c) at (1,0) {};
\node[tiny_vertex] (d) at (2,0) {};
\node[tiny_vertex] (e) at (0,1) {};
\node[tiny_vertex] (f) at (1,1) {};
\draw[thick] (a)--(b)--(c)--(f)--(e)--(b);
\draw[thick] (c)--(d);
\end{tikzpicture}
\end{array}\hspace{-3pt}\right)
\\&\phantom{=}
+q^4x^4f_kf_{k+2}f_{k+2}
\left(\hspace{-8pt}\begin{array}{c}
\begin{tikzpicture}[scale=0.6]
\node[tiny_vertex] (a) at (-1,0) {};
\node[tiny_vertex] (b) at (0,0) {};
\node[tiny_vertex] (c) at (1,0) {};
\node[tiny_vertex] (e) at (0,1) {};
\node[tiny_vertex] (f) at (1,1) {};
\node[tiny_vertex] (h) at (1,-1) {};
\draw[thick] (a)--(b)--(c)--(f)--(e)--(b);
\draw[thick] (c)--(h);
\end{tikzpicture}
\end{array}\hspace{-3pt}\right)
-q^6x^6f_kf_{k+2}f_{k+1}
\left(\hspace{-8pt}\begin{array}{c}
\begin{tikzpicture}[scale=0.6]
\node[tiny_vertex] (a) at (-1,0) {};
\node[tiny_vertex] (b) at (0,0) {};
\node[tiny_vertex] (c) at (1,0) {};
\node[tiny_vertex] (e) at (0,1) {};
\node[tiny_vertex] (f) at (1,1) {};
\draw[thick] (a)--(b)--(c)--(f)--(e)--(b);
\end{tikzpicture}
\end{array}\hspace{-3pt}\right)
\end{align*}

We can repeat this same process for Figure~\ref{Fig:OrangeFamily}(b) but in that case we have $16$ end cases as all four paths can have Corollary~\ref{Cor:fkRecursion} applied.  In the end there are ten possible reduced graphs.  If we combine all the terms together then we get the coefficients shown in Table~\ref{tab:coefficientsBefore}.  

We can apply the recursion from Corollary~\ref{Cor:fkRecursion} to reduce all the $f_i$ terms to combinations of $f_k$ and $f_{k+1}$ (note the $f_2$ terms can be explicitly rewritten in terms of $q$ and $x$).  Doing this for both graphs (remembering to replace each of the graphs shown in Table~\ref{tab:coefficientsBefore} with the appropriate characteristic polynomial) will result in a polynomial expression with terms involving $q,x,f_k,f_{k+1}$.  Using a computer algebra system these were checked to be the same for both graphs, establishing cospectrality.
%
%
%
%
%
\end{proof}

Similar proof techniques using Corollary~\ref{Cor:fkRecursion} and Proposition~\ref{Prop:leaves} can be used to show that similar families of graphs are cospectral. 
\begin{table}[!h]
\begin{center}
\caption{Coefficients for the characteristic polynomial for Proposition~\ref{prop:OrangeFamily}.}
\label{tab:coefficientsBefore}
\begin{tabular}{|c|c|c|}
\hline
     & $G_1$ & $G_2$ \\
     \hline
     $\displaystyle\begin{array}{c}\begin{tikzpicture}[scale = 0.5]
\node[vertex] (a1) at (0,0) {};
\node[vertex] (a2) at (1,0) {};
\node[vertex] (a3) at (0,1) {};
\node[vertex] (a4) at (1,1) {};
\draw[thick] 
(a1)--(a2)--(a4)--(a3)--(a1);

\end{tikzpicture}\end{array}$& $\displaystyle\begin{array}{c}
   0
\end{array}$ &
$\displaystyle\begin{array}{c}
   q^8x^8f_{k}f_{k+1}^2
\end{array}$\\
    \hline
    $\displaystyle\begin{array}{c}\begin{tikzpicture}[scale = 0.5]
\node[vertex] (a1) at (0,0) {};
\node[vertex] (a2) at (1,0) {};
\node[vertex] (a3) at (0,1) {};
\node[vertex] (a4) at (1,1) {};
\node[vertex] (b1) at (2, 1) {};
\draw[thick] 
(a1)--(a2)--(a4)--(a3)--(a1)
(a4)--(b1);

\end{tikzpicture}\end{array}$& $\displaystyle\begin{array}{c}
   -q^6x^6f_{k}f_{k+1}f_{k+2}
\end{array}$  &
$\displaystyle\begin{array}{c}
   -q^6x^6f_2f_{k}f_{k+1}^2\\
   -2q^2x^2f_{k}f_{k+1}f_{k+2}\\
   -q^6x^6f_{k+1}^3
\end{array}$\\ \hline
    $\displaystyle\begin{array}{c}\begin{tikzpicture}[scale = 0.5]
\node[vertex] (a1) at (0,0) {};
\node[vertex] (a2) at (1,0) {};
\node[vertex] (a3) at (0,1) {};
\node[vertex] (a4) at (1,1) {};
\node[vertex] (b1) at (2, 1) {};
\node[vertex] (c1) at (2, 0) {};
\draw[thick] 
(a1)--(a2)--(a4)--(a3)--(a1)
(a2)--(c1)
(a4)--(b1);
\end{tikzpicture}\end{array}$ & $\displaystyle\begin{array}{c}
   q^4x^4f_{k}f_{k+2}^2\\
   q^4x^4f_{k}f_{k+1}f_{k+3}
\end{array}$ & 
$\displaystyle\begin{array}{c}
   q^4x^4f_{k+1}^2f_{k+2}\\
   q^4x^4f_2f_{k}f_{k+1}f_{k+2}\\
   q^4x^4f_{k}f_{k+2}^2
\end{array}$\\ \hline
$\displaystyle\begin{array}{c}
\begin{tikzpicture}[scale = 0.5]
\node[vertex] (a1) at (0,0) {};
\node[vertex] (a2) at (1,0) {};
\node[vertex] (a3) at (0,1) {};
\node[vertex] (a4) at (1,1) {};
\node[vertex] (c1) at (2, 0) {};
\node[vertex] (d1) at (1, -1) {};
\draw[thick]
(a1)--(a2)--(a4)--(a3)--(a1)
(d1)--(a2)--(c1);
\end{tikzpicture}\end{array}$ & 
$\displaystyle\begin{array}{c}
   0
\end{array}$& $\displaystyle\begin{array}{c}
   q^4x^4f_{2}f_{k}f_{k+1}f_{k+2}
\end{array}$
\\\hline
$\displaystyle\begin{array}{c}\begin{tikzpicture}[scale = 0.5]
\node[vertex] (a1) at (0,0) {};
\node[vertex] (a2) at (1,0) {};
\node[vertex] (a3) at (0,1) {};
\node[vertex] (a4) at (1,1) {};
\node[vertex] (a5) at (-1,0) {};
\node[vertex] (b1) at (2, 1) {};
\draw[thick] 
(a5)--(a1)--(a2)--(a4)--(a3)--(a1)
(a4)--(b1);
\end{tikzpicture}\end{array}$ & $\displaystyle\begin{array}{c}
   q^4x^4f_{k+1}^2f_{k+2}
\end{array}$ & $\displaystyle\begin{array}{c}
   q^4x^4f_{2}f_{k+1}^3\\
   q^4x^4f_{k+1}^2f_{k+2}
\end{array}$\\ \hline
$\displaystyle\begin{array}{c}\begin{tikzpicture}[scale = 0.5]
\node[vertex] (a1) at (0,0) {};
\node[vertex] (a2) at (1,0) {};
\node[vertex] (a3) at (0,1) {};
\node[vertex] (a4) at (1,1) {};
\node[vertex] (a5) at (-1,0) {};
\node[vertex] (b1) at (2, 1) {};
\node[vertex] (d1) at (1, -1) {};
\draw[thick] 
(a5)--(a1)--(a2)--(a4)--(a3)--(a1)
(d1)--(a2)
(a4)--(b1);
\end{tikzpicture}\end{array}$ & $\displaystyle\begin{array}{c}
   -q^2x^2f_{k+1}f_{k+2}^2 \\
   -q^2x^2f_{k+1}^2f_{k+3}
\end{array}$ &
$\displaystyle\begin{array}{c}
   -q^2x^2f_2f_{k+1}^2f_{k+2} \\
   -q^2x^2f_{k+1}f_{k+2}^2
\end{array}$\\ \hline
$\displaystyle\begin{array}{c}\begin{tikzpicture}[scale = 0.5]
\node[vertex] (a1) at (0,0) {};
\node[vertex] (a2) at (1,0) {};
\node[vertex] (a3) at (0,1) {};
\node[vertex] (a4) at (1,1) {};
\node[vertex] (a5) at (-1,0) {};
\node[vertex] (c1) at (2, 0) {};
\node[vertex] (d1) at (1, -1) {};
\draw[thick] 
(a5)--(a1)--(a2)--(a4)--(a3)--(a1)
(d1)--(a2)--(c1);
\end{tikzpicture}\end{array}$  & $\displaystyle\begin{array}{c}-q^2x^2f_{k}f_{k+2}f_{k+3}\end{array}$ & $\displaystyle\begin{array}{c}-q^2x^2 f_2 f_k f_{k+2}^2\end{array}$\\
    \hline
$\displaystyle\begin{array}{c}\begin{tikzpicture}[scale = 0.5]
\node[vertex] (a1) at (0,0) {};
\node[vertex] (a2) at (1,0) {};
\node[vertex] (a3) at (0,1) {};
\node[vertex] (a4) at (1,1) {};
\node[vertex] (a5) at (-1,0) {};
\node[vertex] (b1) at (2, 1) {};
\node[vertex] (d1) at (0, -1) {};

\draw[thick] 
(a5)--(a1)--(a2)--(a4)--(a3)--(a1)--(d1)
(a4)--(b1);
\end{tikzpicture}\end{array}$ & $\displaystyle\begin{array}{c}0\end{array}$ & $\displaystyle\begin{array}{c}-q^2x^2f_2f_{k+1}^2f_{k+2}\end{array}$\\
    \hline
$\displaystyle\begin{array}{c}\begin{tikzpicture}[scale = 0.5]
\node[vertex] (a1) at (0,0) {};
\node[vertex] (a2) at (1,0) {};
\node[vertex] (a3) at (0,1) {};
\node[vertex] (a4) at (1,1) {};
\node[vertex] (a5) at (-1,0) {};
\node[vertex] (b1) at (2, 1) {};
\node[vertex] (c1) at (2, 0) {};
\node[vertex] (d1) at (1, -1) {};
\draw[thick] 
(a5)--(a1)--(a2)--(a4)--(a3)--(a1)
(d1)--(a2)--(c1)
(a4)--(b1);
\end{tikzpicture}\end{array}$ & $\displaystyle\begin{array}{c}f_{k+1}f_{k+2}f_{k+3}\end{array}$ & $\displaystyle\begin{array}{c}0\end{array}$\\
    \hline
$\displaystyle\begin{array}{c}\begin{tikzpicture}[scale = 0.5]
\node[vertex] (a1) at (0,0) {};
\node[vertex] (a2) at (1,0) {};
\node[vertex] (a3) at (0,1) {};
\node[vertex] (a4) at (1,1) {};
\node[vertex] (a5) at (-1,0) {};
\node[vertex] (b1) at (2, 1) {};
\node[vertex] (c1) at (2, 0) {};
\node[vertex] (d1) at (0, -1) {};
\draw[thick] 
(a5)--(a1)--(a2)--(a4)--(a3)--(a1)--(d1)
(a2)--(c1)
(a4)--(b1);
\end{tikzpicture}\end{array}$ & $\displaystyle\begin{array}{c}0\end{array}$ & $\displaystyle\begin{array}{c}f_2f_{k+1}f_{k+2}^2\end{array}$\\
    \hline
\end{tabular}
\end{center}
\end{table}

\section{Concluding remarks}\label{sec:future}
We have looked at some basic properties of the exponential distance matrix.  Including some information that can be gathered from the characteristic polynomial which involves coefficients in $q$.

One future direction for research is what happens if we fix a value for $q$ (this was done in special cases in Proposition~\ref{prop:distpoly} and Lemma~\ref{lem:comp}).  We note the following.

\begin{prop}
If $G$ and $H$ are $\Dq$-cospectral for a transcendental value of $q$, then they are cospectral for all values of $q$.
\end{prop}
\begin{proof}
Let $q^*$ be the transcendental value for which the graphs are $\Dq$-cospectral.  Then the coefficient of $x^k$ in $P_{\Dq,G}(x)-P_{\Dq,H}(x)$ is a polynomial in $q$ with $q^*$ as a root.  But because $q^*$ is transcendental this is only possible if the polynomial of $x^k$ is $0$, and since this hold for all $k$ we have $P_{\Dq,G}(x)-P_{\Dq,H}(x)=0$.
\end{proof}

The preceding can also be modified to state that if two graphs are $\Dq$-cospectral for sufficiently many values of $q$ (at least $n$ times the diameter is sufficient), then they are cospectral for all values of $q$.

Conversely, there are some graphs that are cospectral for some values of $q$ but not all (see Figure~\ref{fig: specific q} and Figure~\ref{fig:specific q 2}).
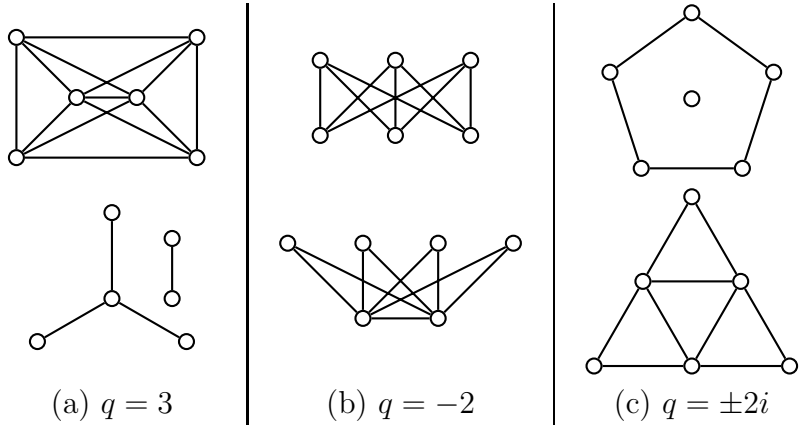
\begin{figure}[!h]
    \centering
    \begin{tabular}{c|c|c}
        \begin{tabular}{c}
\begin{tikzpicture}[rotate = 90,scale = 0.8]
\node[vertex](1) at (0,0) {};
\node[vertex](2) at (2,0) {};
\node[vertex](3) at (1,-1) {};
\node[vertex](4) at (1,-2) {};
\node[vertex](5) at (0,-3) {};
\node[vertex](6) at (2,-3) {};
\draw[thick](1) -- (2) (1) -- (3) (1) -- (4) (1) -- (5) (2) -- (3) (2) -- (4) (2) -- (6) (5) -- (3) (5) -- (4) (5) -- (6) (6) -- (4) (6) -- (3) (3) -- (4);
\end{tikzpicture}

\vspace{10pt}

\\

\begin{tikzpicture}[scale = 0.8]
\node[vertex](1) at (0:0) {};
\node[vertex](2) at (90:1.43) {};
\node[vertex](3) at (210:1.43) {};
\node[vertex](4) at (330:1.43) {};
\node[vertex](5) at (1,0) {};
\node[vertex](6) at (1,1) {};
\draw[thick](1) -- (2) (1) -- (3) (1) -- (4) (5) -- (6);
\end{tikzpicture}
\end{tabular}
&
\begin{tabular}{c}
\begin{tikzpicture}
\node[vertex](1) at (0,0) {};
\node[vertex](2) at (1,0) {};
\node[vertex](3) at (2,0) {};
\node[vertex](4) at (0,-1) {};
\node[vertex](5) at (1,-1) {};
\node[vertex](6) at (2,-1) {};
\draw[thick](1) -- (4) (1) -- (5) (1) -- (6) (2) -- (4) (2) -- (5) (2) -- (6) (3) -- (4) (3) -- (5) (3) -- (6);
\end{tikzpicture}
\vspace{30pt}
\\
\begin{tikzpicture}
\node[vertex](1) at (0,0) {};
\node[vertex](2) at (1,0) {};
\node[vertex](3) at (2,0) {};
\node[vertex](4) at (3,0) {};
\node[vertex](5) at (1,-1) {};
\node[vertex](6) at (2,-1) {};
\draw[thick](1) -- (5) (1) -- (6) (2) -- (5) (2) -- (6) (3) -- (5) (3) -- (6) (4) -- (5) (4) -- (6) (5) -- (6);
\end{tikzpicture}
\end{tabular}
&
\begin{tabular}{c}
\begin{tikzpicture}[scale = 0.8, rotate = 18]
\node[vertex](1) at (0:1.43) {};
\node[vertex](2) at (72:1.43) {};
\node[vertex](3) at (144:1.43) {};
\node[vertex](4) at (216:1.43) {};
\node[vertex](5) at (288:1.43) {};
\node[vertex](6) at (0:0) {};
\draw[thick] (1) -- (2) -- (3) -- (4) -- (5) -- (1);
\end{tikzpicture}
\\
\begin{tikzpicture}[scale = 0.65]
\node[vertex](1) at (0,0) {};
\node[vertex](2) at (-1,{-sqrt(3)}) {};
\node[vertex](3) at (1,{-sqrt(3)}) {};
\node[vertex](4) at (-2,{-2*sqrt(3)}) {};
\node[vertex](5) at (0,{-2*sqrt(3)}) {};
\node[vertex](6) at (2,{-2*sqrt(3)}) {};
\draw[thick](1) -- (2) -- (4) (1) -- (3) -- (6) (2) -- (3) (4) -- (5) -- (6) (2) -- (5) (3) -- (5);
\end{tikzpicture}
\end{tabular}\\
(a) $q = 3$ & (b) $q = -2$ & (c) $q = \pm 2 i$
\end{tabular}
\caption{All pairs of graphs up to six vertices that are only $\Dq$-cospectral for particular $q$.}
\label{fig: specific q}
\end{figure}
\begin{figure}
\centering
\begin{tikzpicture}[scale = 0.6]
        \node[vertex](0) at (0:0) {};
        \node[vertex](1) at (0:3) {};
        \node[vertex](2) at (60:3) {};
        \node[vertex](3) at (120:3) {};
        \node[vertex](4) at (180:3) {};
        \node[vertex](5) at (240:3) {};
        \node[vertex](6) at (300:3) {};
        \draw[thick](0)--(1) (0) -- (2) (0) -- (3) (0) -- (4) (0) -- (5) (0) -- (6)(1) -- (2) (3)--(4) (5)--(6);
    \end{tikzpicture}
    \hfil
    \begin{tikzpicture}[scale = 0.7]
        \node[vertex](0) at (0,0) {};
        \node[vertex](1) at (2,0) {};
        \node[vertex](2) at (0,2) {};
        \node[vertex](3) at (0,-2) {};
        \node[vertex](4) at (-2,-2) {};
        \node[vertex](5) at (-2,0) {};
        \node[vertex](6) at (-2,2) {};
        \draw[thick](0) -- (2) (0) -- (3) (0) -- (4) (0) -- (5) (0) -- (6)(4) -- (5)(3)--(5) (3)--(4) (5)--(6)(2)--(6)(2)--(5);
    \end{tikzpicture}
    \caption{Example of graphs $\Dq$-cospectral for $q = \frac12$.}
    \label{fig:specific q 2}
\end{figure} 

This shows more information getting lost in the spectrum, like the number of components when we fix $q$. It is also possible to have $\mathcal{D}_q^G$ be cospectral with $\mathcal{D}_r^H$ where $q\ne r$ and $G$ and $H$  are non-isomorphic. For example, for $n \ge 3$, the exponential distance matrix of $K_n \sq K_n$ with $q = \frac{n-2}{n-1}$ and the exponential distance matrix of $K_{(n-1)^2+1} \cup (2n-2) K_1$ with $q = \frac{n^2-2n}{(n-1)^2}$ have the same spectrum (see Figure~\ref{fig:different q}).

\begin{figure}
    \centering
\begin{tabular}{c@{\hspace{1in}}c}
\begin{tikzpicture}
\node[vertex](1) at (0:1.43) {};
\node[vertex](2) at (60:1.43) {};
\node[vertex](3) at (120:1.43) {};
\node[vertex](4) at (180:1.43) {};
\node[vertex](5) at (240:1.43) {};
\node[vertex](6) at (300:1.43) {};
\node[vertex](7) at (0:.5) {};
\node[vertex](8) at (120:.5) {};
\node[vertex](9) at (240:.5) {};
\draw[thick](1) -- (2) -- (3) -- (4) -- (5) -- (6) -- (1) (3) -- (9) (3) -- (7) (1) -- (9) (1) -- (8) (2) -- (7) (2) -- (8) (2) -- (7) (4) -- (9) (4) -- (8) (5) -- (7) (5) -- (8) (6) -- (7) (6) -- (9);
\end{tikzpicture}
&
\begin{tikzpicture}
\node[vertex](1) at (18:1.43) {};
\node[vertex](2) at (72+18:1.43) {};
\node[vertex](3) at (144+18:1.43) {};
\node[vertex](4) at (216+18:1.43) {};
\node[vertex](5) at (288+18:1.43) {};
\draw[thick](1) -- (2) -- (3) -- (4) -- (5) -- (1) (1) -- (4) (1) -- (3) (2) -- (5) (2) -- (4) (3) -- (5);
\node[vertex](6) at (2,1.5) {};
\node[vertex](7) at (2,0.5) {};
\node[vertex](8) at (2,-.5) {};
\node[vertex](9) at (2,-1.5) {};
\end{tikzpicture}\\
(a) $C_3\sq C_3, q = \frac{1}{2}$ & (b) $K_3 \cup 4 K_1, q = \frac{3}{4}$
\end{tabular}
\caption{A pair of graphs that have the same $\Dq$-spectrum of $\{4,1,1,1,1,\frac12,\frac12,\frac12,\frac12\}$.}
\label{fig:different q}
\end{figure}
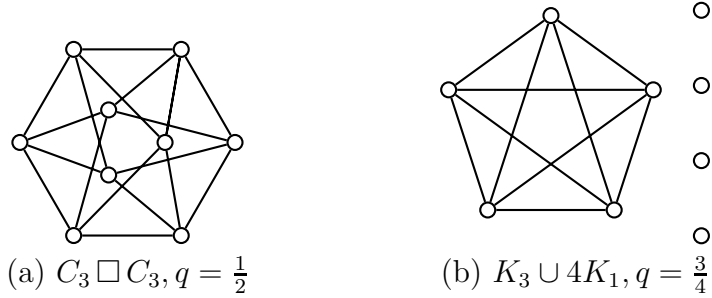

More exploration into what is possible for $q$ fixed or restricted to a few values would be interesting.


Another avenue for exploration is what happens to the spectral radius as we change $q$. There exists graphs where the ordering of the graphs according to spectral radius changes as $q$ changes (see Figure \ref{fig:spectral radius}). Understanding more about this phenomenon would be interesting.

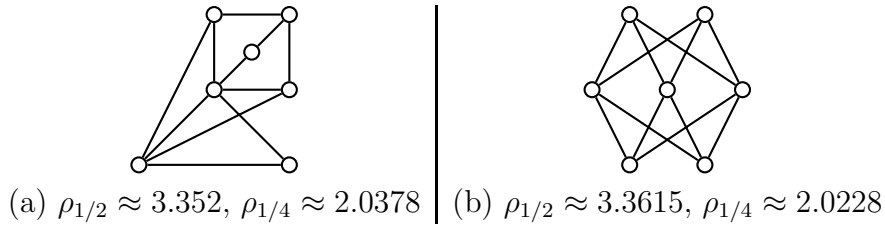
\begin{figure}[!h]
    \centering
    \begin{tabular}{c|c}
         
    \begin{tikzpicture}
\node[vertex](1) at (0,0) {};
\node[vertex](2) at (1,0) {};
\node[vertex](3) at (0.5,-.5) {};
\node[vertex](4) at (0,-1) {};
\node[vertex](5) at (1,-1) {};
\node[vertex](6) at (-1,-2) {};
\node[vertex](7) at (1,-2) {};
\draw[thick](1) -- (2) (1) -- (4) (1) -- (6) (2) -- (3) (2) -- (5) (3) -- (4) (4) -- (5) (4) -- (6) (4) -- (7) (6) -- (7) (5) -- (6);
\end{tikzpicture}
&
\begin{tikzpicture}
\node[vertex](1) at (0,0) {};
\node[vertex](2) at (1,0) {};
\node[vertex](3) at (-.5,-1) {};
\node[vertex](4) at (0.5,-1) {};
\node[vertex](5) at (1.5,-1) {};
\node[vertex](6) at (0,-2) {};
\node[vertex](7) at (1,-2) {};
\draw[thick](1) -- (5) (1) -- (4) (1) -- (3) (2) -- (3) (2) -- (4) (2) -- (5) (4) -- (6) (4) -- (7) (3) -- (6) (3) -- (7) (5) -- (6) (5) -- (7);
\end{tikzpicture}\\
(a) $\rho_{1/2}\approx 3.352$, $\rho_{1/4}\approx 2.0378$ & (b) $\rho_{1/2}\approx 3.3615$, $\rho_{1/4}\approx 2.0228$
\end{tabular}
    \caption{Graphs whose relative ordering for spectral radius, $\rho_q$, changes}
    \label{fig:spectral radius}
\end{figure}

Finally, we believe that there is still more information about a graph that can be derived from its exponential distance matrix.  Also, there are more families of graphs whose spectra can be explicitly computed and cospectral constructions found.  We look forward to seeing more work in this area.

\end{document}